\documentclass[amssymb,amsfonts,reqno]{amsart}
\author{Reuben Wheeler}
\address{University of Edinburgh, JCMB, The King's Buildings, Peter Guthrie Tait Road, Edinburgh, EH9 3FD, Scotland}
\email{reuben.wheeler@ed.ac.uk}
\title{Variation Bounds for Spherical Averages Over Restricted Dilates}

\usepackage{enumerate}
\usepackage[margin=1in]{geometry}
\usepackage{amssymb} 
\usepackage{amsmath} 
\usepackage{amsthm} 
\usepackage{mathtools} 
\usepackage{mathrsfs} 
\usepackage{comment} 
\usepackage{bm} 
\usepackage{bbm}
\usepackage[colorlinks,citecolor=blue,pagebackref,urlcolor=black,hypertexnames=false]{hyperref} 
\usepackage{graphicx} 

\theoremstyle{plain}
   \newtheorem{theorem}[subsubsection]{Theorem}
   \newtheorem*{theorem*}{Theorem}
   
   \newtheorem{conjecture}[subsubsection]{Conjecture}
      
   \newtheorem{remark}[subsubsection]{Remark}

   \newtheorem{proposition}[subsubsection]{Proposition}
   \newtheorem*{proposition*}{Proposition}
   
   \newtheorem{lemma}[subsubsection]{Lemma}

\newcommand{\RR}{\mathbb{R}}

\newcommand{\NN}{\mathbb{N}}

\DeclareMathOperator{\supp}{supp}
\DeclareMathOperator{\conv}{conv}

\DeclareMathOperator{\diam}{diam}

\begin{document}
\begin{abstract}We study $L^p\rightarrow L^q(V^r_E)$ variation semi-norm estimates for the spherical averaging operator, where $E\subset [1,2]$. 
\end{abstract}
\maketitle
We consider the spherical averaging operator
\[\mathcal{A}f(x,t)=\sigma_t*f(x)=\int f(x-ty) d\sigma(y),\]
where $\sigma_t$ denotes the $t$-dilate of the surface measure, $\sigma$, for the sphere. Much effort has been put into studying for which $p,q$ do we have the a priori estimates
\[\|\mathcal{A}f(x,t)\|_{L^q_x(\RR^d;L^\infty_t(\RR^+))}\lesssim \|f\|_{L^p(\RR^d)}?\]
These constitute bounds for the spherical maximal function. The case $p=q$ has been established in important works by Stein in dimension $d\geq 3$ \cite{stein76Spherical} and Bourgain in the case of $d=2$ \cite{bourgain86}: $L^p\rightarrow L^p$ boundedness holds for the maximal operator for $d/(d-1)<p$. Nevertheless, Calderon established that if the spherical maximal function is restricted to a lacunary set of dilates, i.e. we replace $L^\infty_t(\RR^+)$ by $\ell^\infty_t(\mathcal{T})$ for some lacunary $\mathcal{T}\subset \RR^+$, $L^p\rightarrow L^p$ bounds hold for all $p>1$ \cite{calderon79}. The intermediate case was settled by Seeger--Wainger--Wright \cite{seegerWaingerWright95}, revealing the role of Minkowski dimension in the set of dilations. More recently, this work has been extended to the $L^p\rightarrow L^q$ improving case by Anderson--Hughes--Roos--Seeger \cite{andersonHughesRoosSeeger21}, revealing the additional significance of the more locally sensitive notion of Assouad dimension \cite{fraser21}. 

Closely related to maximal function estimates are local variation semi-norm estimates, given as
\[\|V^r_{[1,2]}\left(\mathcal{A}f(x,\cdot)\right)\|_{L^q_x(\RR^d)}\lesssim \|f\|_{L^p(\RR^d)},\]
where $V^r_{[1,2]}$ denotes the local $r$-variation,
$V^r_{[1,2]}F=\sup_{1\leq t_0<t_1<\ldots<t_N\leq 2}\left(\sum_{j=1}^{N-1}\left|F(t_{j+1})-F(t_{j})\right|^r\right)^\frac{1}{r}$, for $1\leq r$ (with the natural modification for $r=\infty$). For any choice of $t_0$, it is apparent that $\sup_{t}|\mathcal{A}f(x,t)|\leq |\mathcal{A}f(x,t_0)| + V^r \left(\mathcal{A}f(x,\cdot)\right)$. As such, variation seminorm estimates can be regarded as stronger than maximal function estimates. These have long been an important area of study in the field \cite{bourgain89} \cite{jonesKaufmanRosenblattWierdl98} \cite{jonesSeegerWright08}. In particular, variation bounds can be used to establish pointwise convergence of averages without the need to establish pointwise convergence on a suitable dense subclass. 

Jones--Seeger--Wright made progress in establishing global (where there is no restriction in the dilations $t\in \RR^+$) variation seminorm estimates for the spherical averaging operator, in the $L^p\rightarrow L^p$ case \cite{jonesSeegerWright08}. Recent work of Beltran--Oberlin--Roncal--Seeger--Stovall \cite{beltranOberlinRoncalSeegerStovall22} established sharp variation semi-norm bounds for the spherical averaging operator, extending this to the $L^p\rightarrow L^q$ case (for the local $t\in [1,2]$ case) and establishing a new endpoint estimate in the $L^p\rightarrow L^p$ case for the global variation. Investigation of the variation operator for spherical averages over restricted sets of dilations is thus a natural question. As in the case of the maximal function, we expect an improved range of $L^p\rightarrow L^q(V_E^r)$ estimates when the dimension of the dilation set $E\subset[1,2]$ is less than $1$. 

Rough estimates for the spherical averaging operator have been considered by Ham--Ko--Lee \cite{hamKoLee22}. Not only does their analysis build on \cite{beltranOberlinRoncalSeegerStovall22}, but they introduce a form of local smoothing for rough $L^p$-averages. This local smoothing for rough averages appears naturally in a more specific form in our analysis.

\subsection{Named exponents}
As a point of reference, we list some of the named exponents required for the statement of our results.
\begin{multline}\\
\frac{1}{q_{\gamma}}=\frac{d-1}{2(d-1+2\gamma)}\\
\frac{1}{q_{\mathrm{LS},\beta}}=\frac{d-1}{2(d-1+\beta)}.\\
\end{multline}

\begin{multline}\\
Q_1=(0,0,0)\\
Q_{2,\beta}=\left(\frac{d-1}{\beta+d-1},\frac{d-1}{\beta+d-1},\frac{d-1}{\beta+d-1}\right)\\
Q_{3,\beta}=\left(\frac{d-\beta}{d+1-\beta},\frac{1}{d+1-\beta},\frac{d-\beta}{d+1-\beta}\right)\\
Q_{4,\gamma}=\left(\frac{d(d-1)}{d^2-1+2\gamma},\frac{d-1}{d^2-1+2\gamma},\frac{d(d-1)}{d^2-1+2\gamma}\right)\\
\end{multline}

The following are basic exponents arising from the above as a simple consequence of $\ell^r$ embedding. 
\begin{multline}\\
\widetilde{Q_{2,\beta}}=\left(\frac{d-1}{\beta+d-1},\frac{d-1}{\beta+d-1},0\right)\\
\widetilde{Q_{3,\beta}}=\left(\frac{d-\beta}{d+1-\beta},\frac{1}{d+1-\beta},0\right)\\
\widetilde{Q_{4,\gamma}}=\left(\frac{d(d-1)}{d^2-1+2\gamma},\frac{d-1}{d^2-1+2\gamma},0\right).\\
\end{multline}

We also name the following
\begin{multline}\\
Q_{5,\beta,\beta}=\left(\frac{d-1}{2(d-1+\beta)},\frac{d-1}{2(d-1+\beta)},\frac{d-1}{2\beta}\right)\\
Q_{6,\beta}=\left(\frac{1}{2},\frac{1}{2},\frac{d-1}{2\beta}\right).\\
\end{multline}

The following exponents are perhaps not as intrinsic to the problem in its general formulation. Indeed, if one is able to obtain appropriate estimates for exponents $r<1$ (see \cite{berghPeetre74}), we no longer expect these to be vertices of associated type sets. Nevertheless, for our purposes in this present work, these are appropriate named exponents. 

\begin{multline}\\
Q_{A,\beta,\beta}=\left(\frac{\beta}{(d-1+\beta)},\frac{\beta}{(d-1+\beta)},1\right)\\ 
Q_{B,\beta}=\left(\frac{d-1-\beta}{d-1},\frac{d-1-\beta}{d-1},1\right)\\
Q_{C,\beta}=\left(\frac{d-2\beta}{d+1-2\beta},\frac{1}{d+1-2\beta},1\right)\\
Q_{D,\beta,\gamma}=\left(\frac{(d-1)(d-\beta)-2\beta\gamma}{d^2-\beta d -1 + \beta+2\gamma-2\beta\gamma},\frac{d-1}{d^2-\beta d -1 + \beta +2\gamma -2\beta \gamma},1\right).
\end{multline}

The exponent $Q_{A,\beta,\beta}$ is related to local smoothing. We don't have the conjectured estimate (see below), except in the case $\beta = 1$ and $d=2$. The following is a basic exponent we use as a substitute, arising from the others by a simple consequence of localisation properties. 
\begin{equation}
Q_{D,\beta,\gamma}^1=\left(\frac{d-1}{d^2-\beta d -1 + \beta+2\gamma-2\beta\gamma},\frac{d-1}{d^2-\beta d -1 + \beta +2\gamma -2\beta \gamma},1\right).
\end{equation}

\subsection{Statement of results}
For a bounded set $E\subset \RR$, let $N(E,\delta)$ denote the $\delta$-covering number of $E$, namely the minimal number of elements in a covering of $E$ by $\delta$-intervals. For a set $E\subset [1,2]$ we denote by $\dim_M E$ the (upper) Minkowski dimension of $E$, namely
\[\inf\left\lbrace s>0 ;\,N(E,\delta)\lesssim \delta^{-s},\, 0<\delta<1\right \rbrace.\]

For a set $E\subset [1,2]$ we denote by $\dim_A E$ the Assouad dimension of $E$, namely
\[\inf\left\lbrace s>0 ;\,\forall I\subset [1,2],\,N(E\cap I,\delta)\leq C_s \left(\frac{\delta}{|I|}\right)^{-s},\, 0<\delta<1\right \rbrace,\]
where $I$ is an interval. 

For a set $E\subset [1,2]$ we denote by $\dim_{A,\theta} E$ the Assouad spectrum of $E$, namely
\[\inf\left\lbrace s>0 ;\, 0<\delta<1,\,\forall I_\theta \subset [1,2],\, N(E\cap I_{\theta},\delta)\leq C_s \left(\delta^{1-\theta}\right)^{-s}\right \rbrace,\]
where $I_\theta$ is an interval of width $\delta^\theta$.

Typically, we use $\beta$ to denote the Minkowski dimension, $\gamma$ to denote Assouad dimension, and $\gamma_\theta$ to denote Assouad spectrum.

\begin{theorem} \label{thm:main}Let $E\subset [1,2]$ be such that $\dim_A E=\gamma$ and $\dim_M E = \beta$. With $d\geq 4$, or $d=3$ and $\gamma \leq -1+\sqrt{3}$, or $d=2$ and $\gamma \leq (-3+\sqrt{17})/4$, suppose that 
 \[\left(\frac{1}{p},\frac{1}{q},\frac{1}{r}\right)\in\left(\conv \left\lbrace Q_1,  \widetilde{Q_{2,\beta}}, \widetilde{Q_{3,\beta}}, \widetilde{Q_{4,\gamma}},Q_{2,\beta}, Q_{3,\beta}, Q_{4,\beta}, Q_{D,\beta,\gamma}, Q_{B,\beta}, Q_{C,\beta},Q_{D,\beta,\gamma}^1\right \rbrace\right)^\circ.\]
Then 
\[\left\|V^r_E\mathcal{A}f\right\|_{L^q}\lesssim \|f\|_{L^p}.\]
\end{theorem}
\begin{figure}
    \centering
    \includegraphics[width=0.75\linewidth]{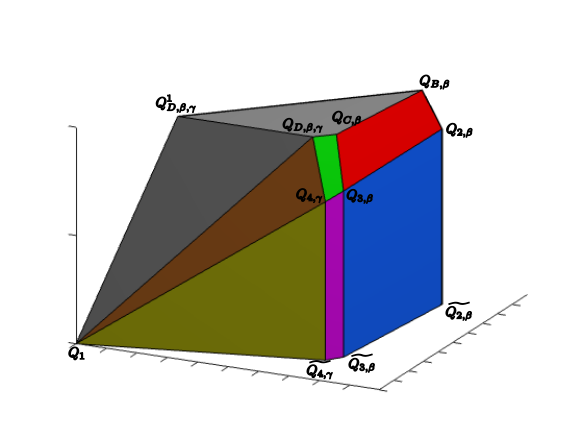}
    \caption{Theorem \ref{thm:main} for $d=4$, $\beta=\gamma=0.7$.}
    \label{fig:thmmain}
\end{figure}

\begin{remark}
Sharpness of Theorem \ref{thm:main} is investigated in Section \ref{sec:sharp}, but we here note some features.

For any $E$ and $d\geq 4$ (or for $\gamma$ sufficiently small in the case $d<4$), in the $(1/p,1/q,1/r)$-region lying vertically above $(1/p,1/q)\in \conv\lbrace (0,0),(1,0),(1/2,1/q_{\gamma})\rbrace$, this results is sharp (up to endpoints). This corresponds to the orange face in Figure \ref{fig:thmmain}. In the case $\beta = \gamma$, in the $(1/p,1/q,1/r)$-region lying vertically above $(1/p,1/q)\in \conv\lbrace (1/2,1/2),(1,0),(1/2,1/q_{\gamma})\rbrace$, the result is sharp (up to endpoints). This corresponds to the green face in Figure \ref{fig:thmmain}. In the $(1/p,1/q,1/r)$-region lying vertically above $(1/p,1/q)\in \conv\lbrace (1,0),(1,1),(1/2,1/2)\rbrace$, sharpness remains undetermined (although we attempt to describe what a sharp example might look like). This region corresponds to the red face in Figure \ref{fig:thmmain}.

Necessary conditions for boundedness of the averaging operator which are independent of $r$ correspond with the maximal function, which is more well understood. We refer to \cite{andersonHughesRoosSeeger21} and, in particular, to \cite{roosSeeger23} for a compelling investigation of these matters. The corresponding faces are the yellow, pink, and blue faces in Figure \ref{fig:thmmain}.
\end{remark}

\begin{remark}
As in \cite{roosSeeger23}, it is to be expected there are interesting geometries of the type set (i.e. the $(1/p,1/q,1/r)$ region corresponding to the range of boundedness for the spherical averaging operator) that more specific $E$ can realise. Understanding the geometry of the type set for a given $E$ and dimension $d$ is an interesting problem. The related question for the maximal function has been investigated by Roos and Seeger \cite{roosSeeger23}, where suitable $E$ can be constructed to realise many possible convex geometries of the type set. Importantly, these authors also extend the result of \cite{andersonHughesRoosSeeger21} to the case $d=2$ and $\gamma> 1/2$. 
\end{remark}



We now state a version of local smoothing for rough time averages, we will explore this conjecture further in future work. 
\begin{conjecture}\label{conj:roughLS}
Suppose that $E\subset [1,2]$ is such that $\dim_M E = \dim_A E = \beta$. Let $\nu\in\mathcal{Z}_j(E)$ index a minimal set of $\delta$-separated $t_\nu$ such that the intervals $[t_\nu,t_\nu+2^{-j}]$ cover $E$. Then, for $0\leq h \leq 2^{-j}$, 
\[\left\|\mathcal{A}_{j}f(x,t_\nu+h)\right\|_{L^{q_{\mathrm{LS},\beta}}_x(\RR^d;\ell^{q_{\mathrm{LS},\beta}}_\nu(\mathcal{Z}_j(E)))}\lessapprox 2^{-\frac{j(d-1)}{2}}|\mathcal{Z}_j(E)|^{1/ q_{\mathrm{LS},\beta}}\|f\|_{L^{q_{\mathrm{LS},\beta}}}.\]
\end{conjecture}
In fact, the conjecture is perhaps more generally stated with respect to a rough measure (see Appendix \ref{sec:LS}, where we show Conjecture \ref{conj:roughLSMeasure} implies Conjecture \ref{conj:roughLS} for Ahlfors--David regular sets).
\begin{conjecture}\label{conj:roughLSMeasure}
Suppose that $\mu$ is an Ahlfors--David regular probability measure supported on $[1,2]$ such that $\mu(B(t,r))\sim r^\beta$ for $r<\diam\left(\supp\mu\right)$. Then
\[\left\|\left(e^{\pm it\sqrt{-\Delta}}\phi_j\left(\sqrt{-\Delta}\right)f\right)(x)\right\|_{L^{q_{\mathrm{LS},\beta}}_x(\RR^d;L^{q_{\mathrm{LS},\beta}}_t(d\mu))}\lessapprox \|f\|_{L^{q_{\mathrm{LS},\beta}}}.\]
\end{conjecture}

\begin{remark}We might also consider the conjecture for less regular measures. For example $\mu$ supported on $[1,2]$ such that $\mu(B(t,r))\lesssim r^\beta$ with Assouad dimension $\gamma$. Then, according with a Knapp example, we may expect the estimate
\begin{equation}\label{conj:roughLSMeasureAss}\left\|\left(e^{\pm it\sqrt{-\Delta}}\phi_j\left(\sqrt{-\Delta}\right)f\right)(x)\right\|_{L^{q_{\mathrm{LS},\gamma}}_x(\RR^d;L^{q_{\mathrm{LS},\gamma}}_t(d\mu))}\lessapprox \|f\|_{L^{q_{\mathrm{LS},\gamma}}}.\end{equation}
It is likely the full range local smoothing estimates we can obtain will be sensitive to various dimensional notions regarding $\mu$ and we don't expect interpolating \eqref{conj:roughLSMeasureAss} with the trivial $L^2$ estimate would give sharp bounds in general.
\end{remark}

\begin{theorem}\label{thm:LSimplies}Let $E\subset [1,2]$ be such that $\dim_A E=\gamma=\dim_M E = \beta$. If $d=2$ suppose, additionally, that $\gamma\leq \frac{1}{2}$. Suppose that Conjecture \ref{conj:roughLS} holds, and
 \[\left(\frac{1}{p},\frac{1}{q},\frac{1}{r}\right)\in\left(\conv \left\lbrace Q_1,  \widetilde{Q_{2,\beta}}, \widetilde{Q_{3,\beta}}, \widetilde{Q_{4,\gamma}},Q_{2,\beta}, Q_{3,\beta}, Q_{4,\beta}, Q_{D,\beta,\gamma}, Q_{B,\beta}, Q_{C,\beta},Q_{A,\beta,\gamma}\right \rbrace\right)^\circ,\]
then
\[\left\|V^r_E\mathcal{A}\left(f(x,\cdot)\right)\right\|_{L^q_x}\lesssim \|f\|_{L^p}.\]
\end{theorem}
\begin{figure}
    \centering
    \includegraphics[width=0.75\linewidth]{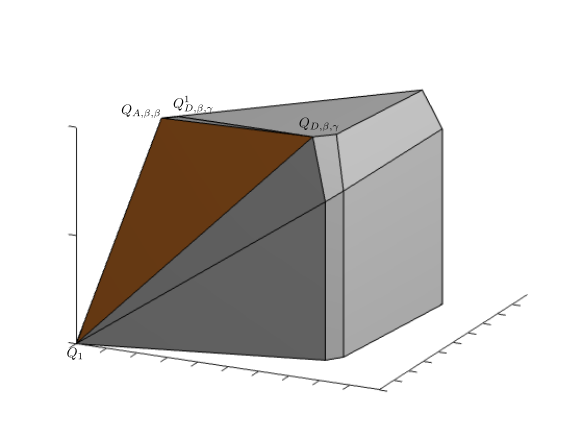}
    \caption{Theorem \ref{thm:LSimplies} relative to Theorem \ref{thm:main} for $d=4$, $\beta=\gamma=0.7$.}
    \label{fig:thmLS}
\end{figure}

\begin{remark}
Unlike \cite{beltranOberlinRoncalSeegerStovall22}, our arguments require the local smoothing estimate in \emph{all} dimensions (not just $d=2$ and $d=3$) to obtain sharp estimates. In their case, $\beta=\gamma = 1$ and $Q_{D,1,1}^1=Q_{A,1,1}$. However, when $\beta<1$, there is a more significant loss arising from the localisation properties that transfer estimates from $Q_{D,\beta,\gamma}$ to $Q_{D,\beta,\gamma}^1$ (see Section \ref{sec:multiShell}: the vertex $Q_{D,\beta,\gamma}^1$ does not pass the test with equality for $\beta<1$, but $Q_{A,\beta,\beta}$ does). 
\end{remark}

\begin{remark}
It is of interest what $Q_{A,\beta,\gamma}$ should be when $\beta<\gamma$, however this requires further investigation. 
\end{remark}

\begin{theorem}\label{thm:d2LSimplies}Let $E\subset [1,2]$ be such that $\dim_A E=\gamma=\dim_M E = \beta$. Suppose also that $d=2$ and $\gamma > \frac{1}{2}$. Suppose that Conjecture \ref{conj:roughLS} holds, and
 \[\left(\frac{1}{p},\frac{1}{q},\frac{1}{r}\right)\in\left(\conv \left\lbrace Q_1,  \widetilde{Q_{2,\beta}}, \widetilde{Q_{3,\beta}}, \widetilde{Q_{4,\gamma}},Q_{2,\beta}, Q_{3,\beta}, Q_{4,\beta}, Q_{5,\beta,\beta}, Q_{6,\beta}\right \rbrace\right)^\circ,\]
then
\[\left\|V^r_E\mathcal{A}\left(f(x,\cdot)\right)\right\|_{L^q_x}\lesssim \|f\|_{L^p}.\]
\end{theorem}
\begin{figure}
    \centering
    \includegraphics[width=0.7\linewidth]{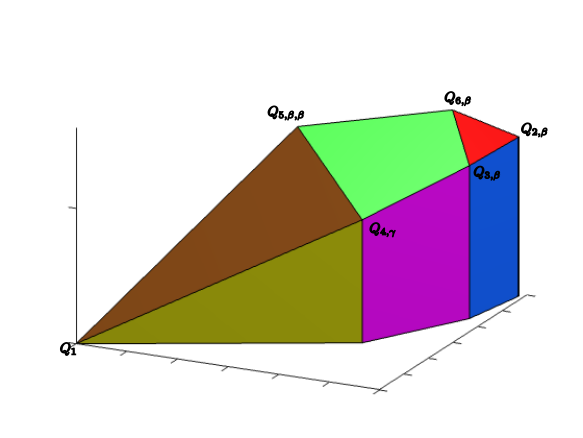}
    \caption{Theorem \ref{thm:d2LSimplies} when $\beta=\gamma=0.7$.}
    \label{fig:d2LS}
\end{figure}

\section*{Acknowledgements}
The author would like to express gratitude to his PhD supervisor Jim Wright, by whom he was first introduced to the spherical maximal function over restricted dilates. The author would like to thank Alex Rutar for inspiring conversations about the work. The author would also like to thank Luz Roncal for advice on the paper. The author is also grateful to Richard Oberlin, for his advice on how to render 3D images using Matlab.

\section{Notation}
For a set $S\subset\RR^3$, we denote by $\conv S$ the convex hull of these points. We denote by $S^\circ$ the interior of $S$.

We denote by $\supp f$ or $\supp \mu$ the support of a function or measure. 

Throughout, $\sigma$ will be used to denote the surface measure of the unit sphere and $\sigma_t$ the $t$-dilate thereof.

We use an indexed mixed norm notation for functions of multiple variables. For example,
\[\left\| g(x,\nu)\right\|_{L^q_x(\ell^r_\nu (\mathcal{Z}_j(E)))} = \left(\int_{\RR^d} \left(\sum_{\nu\in \mathcal{Z}_j(E)}|g(x,\nu)|^r\right)^{\frac{q}{r}}dx\right)^{\frac{1}{q}}.\]

If we write $F_j(x)\lesssim G_j(x)$ or $F_j(x)=O(G_j(x))$ for non-negative functions, this means there exists a constant $C$ (independent of $j$) such that
\[F_j(x)\leq CG_j(x).\]
If we write $F_j(x)\lessapprox G_j(x)$, this means for all $\epsilon>0$, there exists a constant $C_\epsilon$ (independent of $j$) such that
\[F_j(x)\leq C_\epsilon 2^{j\epsilon}G_j(x).\]

\section{Outline}
In Section \ref{sec:embedding}, we prove an embedding theorem bounding $L^p\rightarrow L^q(V^r_E)$ operator estimates by mixed norm $L^p\rightarrow L^q(\ell^r)$ estimates. In Section \ref{sec:coreEsts}, we state the core estimates we use in our interpolation argument before going on to prove them. Section \ref{subsec:oscIntEsts} outlines the core estimates obtained via the oscillatory integral representation of the averaging operator. Section \ref{subsec:kernelEsts} contains proofs of these estimates relying on the kernel representation of our operator. Section \ref{sec:interpolation} is where we prove our main results as a simple consequence of our embedding and the core estimates established proved in Section \ref{sec:coreEsts}: Section \ref{sec:localisation} introduces some localisation lemmas, which we combine with interpolation in Section \ref{subsec:endpointInterpolation} to prove Theorems \ref{thm:main}, \ref{thm:LSimplies}, and \ref{thm:d2LSimplies}. Section \ref{sec:sharp} includes some examples relating to the sharpness of our results.

In Appendix \ref{sec:LS} we explore the problem of local smoothing for rough time-averages, it is included as context for the interested reader to relate Conjectures \ref{conj:roughLS} and \ref{conj:roughLSMeasure}, but the results of this paper do not rely on it. We include a proof of the Besov embedding for $V^r_{[1,2]}$ in Appendix \ref{app:besov}.

\section{An embedding theorem}\label{sec:embedding}
It is natural to make a Littlewood--Paley decomposition of our operator. For suitable bump functions $\sum_{j\geq 0}\phi_j(|\xi|)=1$, with $\phi_j(|\xi|)=\phi_1(2^{-j+1}|\xi|)$ for $j\geq 1$, we denote the frquency projection operator $P_j$ such that
\[\widehat{P_jf}(\xi)=\phi_j(|\xi|)\hat{f}(\xi).\]
We then seek summable operator estimates for the operators $\mathcal{A}_j=\mathcal{A}\circ P_j$.

Motivated by the proof of the Besov embedding (see Appendix \ref{app:besov}), we have the following result. It can be observed that this is simply a strengthening of the maximal function embedding utilised in \cite{andersonHughesRoosSeeger21}.
\begin{theorem}For each scale $\delta=2^{-j}$, index by $\nu\in\mathcal{Z}_j(E)$ the set of left endpoints $t_\nu$ of a minimal $\delta$-cover of $E\subset [1,2]$ such that the left endpoints are $\delta$-separated. Then we have that, for $1\leq r\leq q$,
\[\|V^r_E(\mathcal{A}_jf(x,\cdot))\|_{L^q_x}\lesssim \left\|\mathcal{A}_{j}f(x,t_\nu)\right\|_{L^q_x(\RR^d;\ell^r_\nu(\mathcal{Z}_j(E)))}+\sup_{h\in [0,2^{-j}]}\left\|2^{-j}\frac{\partial}{\partial t}\mathcal{A}_{j}f(x,t_\nu+h)\right\|_{L^q_x(\RR^d;\ell^r_\nu(\mathcal{Z}_j(E)))}.\]
\end{theorem}
\begin{proof}
We first note that that, for some choice of $t_{1} < t_{2} < \ldots < t_{N}$ in $E$,
\[V_E^r\left(\mathcal{A}_{j}f\right)(x)\]
\[\lesssim \left(\sum_{i=2}^N\left|\mathcal{A}_{j}f(x,t_{i})-\mathcal{A}_{j}f(x,t_{i-1})\right|^r\right)^{\frac{1}{r}}.\] 
Now, let $\mathcal{I}_l=\lbrace I_1,I_2,\ldots I_J\rbrace$ be an ordered cover of $E$ by $\delta$-intervals, with $\delta=2^{-l}$. We can assume without loss of generality that the left endpoints of the intervals are $\delta$-separated. Suppose that we have a sequence $t_1<t_2<\ldots<t_N$ of points in $E$. Starting from $n_0=0$, we find $n_1,n_2,\ldots,n_J$ and $N(k)=\sum_{j=1}^{k}n_j$ such that for $N(k-1)\leq i < N(k)$, $t_i\in I_k=[a_k,a_k+\delta]$. We then rewrite the previously displayed expression as
\[
=\left(\sum_{k=1}^{J}\sum_{i=N(k-1)+1}^{N(k)}|\mathcal{A}_{j}f(t_i)-\mathcal{A}_{j}f(t_{i-1})|^r+\sum_{k=2}^{J}|\mathcal{A}_{j}f(t_{N(k)})-\mathcal{A}_{j}f(t_{N(k)-1})|^r\right)^\frac{1}{r}
\]
\begin{equation}\label{eq:opEmbFTCterms}
\lesssim\left(\sum_{k=1}^{J}\sum_{i=N(k-1)}^{N(k)}|\mathcal{A}_{j}f(x,t_i)-\mathcal{A}_{j}f(x,t_{i-1})|^r\right)^\frac{1}{r}
\end{equation}
\[
+\left(\sum_{k=2}^{J-1}|\mathcal{A}_{j}f(x,t_{N(k)})-\mathcal{A}_{j}f(x,a_{k+1})|^r\right)^\frac{1}{r}
\]
\[
+\left(\sum_{k=2}^{J-1}|\mathcal{A}_{j}f(x,a_{k})-\mathcal{A}_{j}f(x,t_{N(k)-1})|^r\right)^\frac{1}{r}
\]
\begin{equation}\label{eq:opEmbPPterms}
+\left(\sum_{k=1}^{J}|\mathcal{A}_{j}f(x,a_{k+1})|^r\right)^\frac{1}{r}.
\end{equation}

It is most instructive to consider the sum in \eqref{eq:opEmbFTCterms}, the other difference terms can be treated analogously. As in the Besov embedding, we use the fundamental theorem of calculus and H\"{o}lder's inequality to find that 
\[\left(\sum_{k=1}^{J}\sum_{i=N(k-1)}^{N(k)}|\mathcal{A}_{j}f(x,t_i)-\mathcal{A}_{j}f(x,t_{i-1})|^r\right)^\frac{1}{r}
\]
\[\lesssim\left(\sum_{k=1}^{J}\sum_{i=N(k-1)}^{N(k)}(t_i-t_{i-1})^{r-1}\int_{t_{i-1}}^{t_i}\left|\frac{\partial}{\partial t}\mathcal{A}_{j}f(x,t)\right|^rdt\right)^\frac{1}{r}
\]
\[\lesssim \delta^{1-\frac{1}{r}}\left(\sum_{k=1}^{J}\int_{I_k}\left|\frac{\partial}{\partial t}\mathcal{A}_{j}f(x,t)\right|^rdt\right)^\frac{1}{r}
\]
\[\lesssim \delta^{1-\frac{1}{r}}\left(\int_{0}^{\delta}\sum_{\nu\in \mathcal{Z}_j}\left|\frac{\partial}{\partial t}\mathcal{A}_{j}f(x,t_\nu + h)\right|^rdh\right)^\frac{1}{r}.
\]
We are interested in the $L^q_x$ norm of the previous expression. By an application of Minkowski's inequality relative to the $L^{q/r}_x(L^1_h)$ expression, followed by H\"older's inequality, we obtain the bound
\[\left\|V^r\left(\mathcal{A}_j f(x,\cdot)\right)\right\|_{L^q_x}\lesssim \left\|\mathcal{A}_{j}f(x,t_\nu)\right\|_{L^q_x(\RR^d;\ell^r_\nu(\mathcal{Z}_j(E)))}+\delta^{-\frac{1}{r}}\left(\int_0^{\delta}\left\|\delta\frac{\partial}{\partial t} \mathcal{A}_jf (x,t_\nu+h)\right\|_{L^q_x(\RR^d;\ell^r_\nu(\mathcal{Z}_j(E)))}^{r}dh\right)^{\frac{1}{r}}\]
\[\lesssim \left\|\mathcal{A}_{j}f(x,t_\nu)\right\|_{L^q_x(\RR^d;\ell^r_\nu(\mathcal{Z}_j(E)))}+\sup_{h\in [0,\delta]}\left\|\delta\frac{\partial}{\partial t} \mathcal{A}_jf (x,t_\nu+h)\right\|_{L^q_x(\RR^d;\ell^r_\nu(\mathcal{Z}_j(E)))}.\]
\end{proof}

\section{Core operator estimates}\label{sec:coreEsts}
We seek summable bounds 
\[\sup_{h\in [0, 2^{-j}]}\left\|\mathcal{A}_{j}f(x,t_\nu+h)\right\|_{L^q_x(\RR^d;\ell^r_\nu(\mathcal{Z}_j(E)))}\lesssim C_{j,p,q}\|f\|_{L^p}\]
and
\[ \sup_{h\in [0,2^{-j}]}\left\|\delta\frac{\partial}{\partial t} \mathcal{A}_jf (x,t_\nu+h)\right\|_{L^q_x(\RR^d;\ell^r_\nu(\mathcal{Z}_j(E)))}\lesssim C_{j,p,q}\|f\|_{L^p}.\]
These are obtained by interpolation, $\ell^r$ embedding, and H\"older's inequality. We here state the endpoint estimates utilised for the interpolation argument. 

Writing the estimates for $\mathcal{A}_j$ rather than $2^{-j} \frac{\partial}{\partial t} \mathcal{A}_j$, which by our analysis (see Remark \ref{rmk:partialSameAsA}), are expressed equivalently. The core operator estimates are as follows. We show, for all $h\in [0,2^{-j}]$,
\begin{equation}\label{eq:core1inf1}\left\|\mathcal{A}_{j}f(x,t_\nu+h)\right\|_{L^\infty_x(\RR^d;\ell^1_\nu(\mathcal{Z}_j(E)))}\lesssim 2^{j}\|f\|_{L^1},\end{equation}
\begin{equation}\label{eq:core111}\left\|\mathcal{A}_{j}f(x,t_\nu+h)\right\|_{L^1_x(\RR^d;\ell^1_\nu(\mathcal{Z}_j(E)))}\lesssim |\mathcal{Z}_j(E)|\|f\|_{L^1},\end{equation}
\begin{equation}\label{eq:coreinfinfinf}\left\|\mathcal{A}_{j}f(x,t_\nu+h)\right\|_{L^\infty_x(\RR^d;\ell^\infty_\nu(\mathcal{Z}_j(E)))}\lesssim \|f\|_{L^\infty},\end{equation}
and 
\begin{equation}\label{eq:core222}\left\|\mathcal{A}_{j}f(x,t_\nu+h)\right\|_{L^2_x(\RR^d;\ell^2_\nu(\mathcal{Z}_j(E)))}\lesssim |\mathcal{Z}_j(E)|^{1/2}2^{-j\frac{(d-1)}{2}}\|f\|_{L^2}.\end{equation}
The following estimate is an improvement to the one found in \cite{andersonHughesRoosSeeger21} and we adapt their argument. For $d\geq 3$ or  $\gamma\leq \frac{1}{2}$ in the case $d=2$,
\begin{equation}\label{eq:core2q2}\left\|\mathcal{A}_{j}f(x,t_\nu+h)\right\|_{L^{q_{\gamma}}_x(\RR^d;\ell^{2}_\nu(\mathcal{Z}_j(E)))}\lessapprox 2^{-j\frac{(d-1)^2-2\gamma}{2(d-1+2\gamma)}}\|f\|_{L^2},\end{equation} 
where $q_{\gamma}=\frac{2(d-1+2\gamma)}{d-1}$. Note this corresponds with the Stein--Tomas restriction exponent when $\gamma=1$.
\begin{remark}
The author has learned that \eqref{eq:core2q2} has been obtained independently by David Beltran, Joris Roos, and Andreas Seeger. In fact, if $N(E\cap I, \delta)\lesssim (\delta/|I|)^{-\gamma}$, they also show (without an $\epsilon$-loss)
\[\left\|\mathcal{A}_{j}f(x,t_\nu+h)\right\|_{L^{q_{\gamma},\infty}_x(\RR^d;\ell^{2}_\nu(\mathcal{Z}_j(E)))}\lesssim 2^{-j\frac{(d-1)^2-2\gamma}{2(d-1+2\gamma)}}\|f\|_{L^{2}}.\]
\end{remark}

With $q_{\mathrm{LS},\beta}=\frac{2(d-1+\beta)}{d-1}$, we express a conjecture of local smoothing for rough time-averages:
\begin{equation}\label{eq:coreqqq}\left\|\mathcal{A}_{j}f(x,t_\nu+h)\right\|_{L^{q_{\mathrm{LS},\beta}}_x(\RR^d;\ell^{q_{\mathrm{LS},\beta}}_\nu(\mathcal{Z}_j(E)))}\lessapprox 2^{-\frac{j(d-1)}{2}}|\mathcal{Z}_j(E)|^{1/ q_{\mathrm{LS},\beta}}\|f\|_{L^{q_{\mathrm{LS},\beta}}}.\end{equation} 
The exponent $q_{\mathrm{LS},\beta}$ is given by $q_{\mathrm{LS},\beta}=\frac{2(d-1+\beta)}{d-1}$. Note that it corresponds with the classical local smoothing exponent when $\beta=1$. This problem appears in \cite{hamKoLee22} in a more general form, where they obtain partial results.

\subsection{Oscillatory integral operator bounds}\label{subsec:oscIntEsts}
Taking the Fourier transform of $\mathcal{A}_jf(\cdot,t)$, we have
\[\widehat{\mathcal{A}_jf(\cdot,t)}(\xi)=\phi_j(|\xi|)\hat{\sigma}(t|\xi|)\hat{f}(\xi).\]
In this way, working by the method of stationary phase to expand $\hat{\sigma}(t|\xi|)$, we can represent, for $t\in [1,2]$ and $j\geq 1$
\[\mathcal{A}_jf(x,t)=\mathcal{A}_j^{+}f(x,t)+\mathcal{A}_j^{-}f(x,t),\] 
where
\[\mathcal{A}_j^{\pm}f(x,t)=\int e^{i \left(x\cdot \xi \pm t |\xi| \right)}\phi_j(|\xi|)a_{\pm}(t,\xi)\hat{f}(\xi)d\xi
,\]
where $a\in S^{-\frac{d-1}{2}}$ is a symbol of order $-(d-1)/2$, i.e. for all multi-indices $\alpha$ and $M\in\NN_0$
\[\left|\partial_t^{M}\partial_\xi^{\alpha}a(t,\xi)\right|\lesssim_{M,\alpha}(1+|\xi|)^{-\frac{(d-1)}{2}-|\alpha|}.\]
See, for instance, Chapter 8 of \cite{stein93}.

In this section, using this representation, we prove the bounds \eqref{eq:core222} and \eqref{eq:core2q2}. For easy reference, we index the estimates by $(1/p,1/q,1/r)$: the estimates are $(1/2,1/2,1/2)$ and $(1/2,1/q_{\gamma},1/2)$. 

\begin{remark}\label{rmk:partialSameAsA} Note that $2^{-j}\frac{\partial}{\partial t} \mathcal{A}_j$ is an Fourier integral operator with a symbol of the same order and support as that of $\mathcal{A}_j$. Although we require bounds for both $2^{-j}\frac{\partial}{\partial t} \mathcal{A}_j$ and $ \mathcal{A}_j$, we present the arguments only for $\mathcal{A}_j$. This is because all estimates we obtain depend on: the support of the symbol, the order of the symbol, and the $\delta$-separation of the points $t_\nu+h$.
\end{remark} 

\begin{proof}[Proof of $(1/2,1/2,1/2)$ estimate \eqref{eq:core222}]
We use an $L^\infty$  Fourier multiplier bound and Plancherel's theorem to see that
\[\left\|\mathcal{A}_{j}f(x,t_\nu)\right\|_{L^2_x(\RR^d;\ell^2_\nu(\mathcal{Z}_j(E)))}\]
\[=\left( \sum_{\nu\in\mathcal{Z}_j(E)}\left\|\widehat{\mathcal{A}_{j}f(\cdot,t_\nu)}\right\|_{L^2}^2 
 \right)^{\frac{1}{2}}\]
\[\lesssim |\mathcal{Z}_j(E)|^{\frac{1}{2}}2^{-\frac{j(d-1)}{2}}\|f\|_{L^2}.\]
\end{proof}

The $Q_{4,\gamma}=(1/2,1/q_\gamma,1/2)$ estimate is established by a $TT^*$ argument. We define by $T_{j,t}^{\pm}$ a Fourier integral operator with symbol in $S^0$ given as
\[T_{j,t}^{\pm}f(x)=\int e^{i\left(x\cdot \xi \pm t|\xi|\right)}\phi_j(\xi)\left(2^{j\frac{(d-1)}{2}}a(t,\xi)\right)\hat{f}(\xi)d\xi.\]
We can thus write $\mathcal{A}_jf(x,t)= 2^{-j(d-1)/2}T_tf(x)$. 
We define 
\[S_{n,j}g(x,\nu) \coloneqq 2^{-j(d-1)} \sum_{\nu'\in\mathcal{Z}_{n,j}(\nu)}(T_{j,t_{\nu}}^{\pm}(T_{j,t_{\nu'}}^{\pm})^*g(\cdot,\nu'))(x),\] 
where $\mathcal{Z}_{n,j}(\nu)$ is the set of all $\nu'\in \mathcal{Z}_{j}(E)$ such that $2^{-j+n-1}\leq |t_\nu-t_{\nu'}|<2^{-j+n}$. They interpolate
\begin{equation}\label{eq:AHRS1inf}\|S_{n,j}g(x,\nu)\|_{L^\infty_{x}(\ell^\infty_\nu)}\lesssim 2^{-\frac{n(d-1)}{2}+j}\|g(x,\nu)\|_{L^1_x(\ell^1_\nu)}\end{equation}
with
\[\left(\int\sum_{\nu}\left|S_{n,j}g(x,\nu)\right|^2dx\right)^{\frac{1}{2}}\lessapprox 2^{n\gamma-j(d-1)}\|g\|_{L^2_x(\ell^2_\nu)}.\]

The $Q_{4,\gamma}=(1/2,1/q_\gamma,1/2)$ estimate \eqref{eq:core2q2} follows directly from their argument by replacing the estimate \eqref{eq:AHRS1inf} in their paper with Lemma \ref{lem:AHRSadapted}. 
\begin{lemma}\label{lem:AHRSadapted}We have that
\[\|S_{n,j}g(x,\nu)\|_{L^{\infty}_x(\ell^2_{\nu})}\lessapprox 2^{j}2^{-n\frac{d-1}{2}}\|g(x,\nu)\|_{L^1_x(\ell^2_\nu)}.\]
\end{lemma}
\begin{proof}
In fact, we can write $T_{j,t_{\nu}}^{\pm}(T_{j,t_{\nu'}}^{\pm})^*f =K_{j,\nu,\nu'}^{\pm}*f$. We have the kernel bounds
\[|K_{j,\nu,\nu'}(x)|\lesssim_{N} 2^{jd}(1+2^j|x|)^{-\frac{d-1}{2}}(1+2^j||x|-|t_\nu-t_{\nu'}||)^{-N},\]
which are uniform over $\nu,\nu'$ (as can be verified by the method of stationary phase).
We have that 
\[S_{n,j}f(x,\nu)=2^{-j(d-1)}\sum_{\nu'\in \mathcal{Z}_{n,j}(\nu)}T_{j,t_{\nu}}^{\pm}(T_{j,t_{\nu'}}^{\pm})^*g(\cdot,\nu)(x).\]

Thus we find that
\[\left(\sum_{\nu}\left|\sum_{\nu'\in \mathcal{Z}_{n,j}(\nu)}T_{j,t_{\nu}}(T_{j,t_{\nu'}}^{\pm})^*g(\cdot,\nu)(x)\right|^2\right)^{\frac{1}{2}}\]
\[\lesssim_{N} \left(\sum_{\nu}\left|\sum_{\nu'\in \mathcal{Z}_{n,j}(\nu)}\int 2^{jd}(1+2^j|y|)^{-\frac{d-1}{2}}(1+2^j||y|-|t_\nu-t_{\nu'}||)^{-N}|g(x-y,\nu)|dy\right|^2\right)^{\frac{1}{2}}.\]
Notice that the $K_{j,\nu,\nu'}(y)$ is substantial only in the case that $||y|-|t_\nu-t_{\nu'}||\leq 2^{-j(1-\epsilon)}$, which can happen only for $|y|\sim 2^{n-j}$ and, for a given $\nu$, $O(2^{j\epsilon\gamma})$ $\nu'$'s. As such, up to a well controlled remainder term we consider momentarily, we have the bound
\[\lesssim 2^{j\epsilon\gamma}2^{jd}2^{-n\frac{d-1}{2}}\left(\sum_{\nu}\left|\int |g(x-y,\nu)|dy\right|^2\right)^{\frac{1}{2}}\]
\[\lesssim 2^{j\epsilon\gamma}2^{jd}2^{-n\frac{d-1}{2}}\int\left(\sum_{\nu}|g(x-y,\nu)|^2\right)^{\frac{1}{2}}dy.\]
For the remainder term, we apply Fubini to then sum first in $\nu'$. In this way, it is bounded by
\[\lesssim_{N} \left(\sum_{\nu}\left|\int_{||y|-|t_\nu-t_{\nu'}||\geq 2^{-j(1-\epsilon)}} 2^{jd}\sup_{\nu'\in \mathcal{Z}_{n,j}(\nu)}(1+2^j||y|-|t_\nu-t_{\nu'}||)^{-N}|g(x-y,\nu)|dy\right|^2\right)^{\frac{1}{2}}\]
\[\lesssim  2^{jd}2^{-n\frac{d-1}{2}} \int \left(\sum_{\nu}|g(y,\nu)|^2\right)^{\frac{1}{2}}dy,\]
for a suitable choice of $N$ depending on $\epsilon$.
\end{proof}

\subsection{The convolution kernel perspective}\label{subsec:kernelEsts}
With \[\mathcal{A}_{j}f(x,t)=2^{-\frac{j(d-1)}{2}}(2\pi)^{-d-1}\sum_{\pm}\int \kappa_{j}^{\pm}(y,t) f(x-y)dy,\]
we have, as in \cite{andersonHughesRoosSeeger21}, the kernel bound
\[|\kappa_{j}^{\pm}(x,t)|\lesssim_{N} 2^{jd}(1+2^j|x|)^{-\frac{d-1}{2}}(1+2^j||x|-|t||)^{-N}.\]
This can be verified simply by polar integration and the method of stationary phase. 

In this section, we prove the estimates \eqref{eq:core1inf1}, \eqref{eq:core111}, \eqref{eq:coreinfinfinf}. Indexed by $(1/p,1/q,1/r)$, these are the $(1,0,1)$,  $(1,1,1)$, and $(0,0,0)$ estimates.

\begin{proof}[Proof of $(1,0,1)$ estimate \eqref{eq:core1inf1}]
We see using the kernel estimate and the Fubini--Tonelli theorem that
\[\left\|\mathcal{A}_{j}f(x,t_\nu)\right\|_{L^\infty_x(\RR^d;\ell^1_\nu(\mathcal{Z}_j(E)))}\]
\[\lesssim_N 2^{-\frac{j(d-1)}{2}}\sup_{x}\int \sum_{\nu\in\mathcal{Z}_j(E)}\left|2^{jd}(1+2^j|y|)^{-\frac{d-1}{2}}(1+2^j||y|-|t_\nu||)^{-N} f(x-y)\right|dy.\]
It is easy to see, since $t_\nu\in [1,2]$, that the kernel is essentially supported for $ y\in A_{\frac{1}{4},4}(0)$.  The $t_\nu$ are $2^{-j}$ separated, so after summation, we have the bound, for $ y\in A_{\frac{1}{4},4}(0)$,
\[\sum_{\nu\in\mathcal{Z}_j(E)}2^{jd}(1+2^j|y|)^{-\frac{d-1}{2}}(1+2^j||y|-|t_\nu||)^{-N}\]
\[\lesssim\sup_{\nu\in\mathcal{Z}_j(E)}2^{jd}(1+2^j|y|)^{-\frac{d-1}{2}}(1+2^j||y|-|t_\nu||)^{-N}\]
\[\lesssim 2^j 2^{j\frac{d-1}{2}}.\]
For $y\notin A_{\frac{1}{4},4}(0)$,
\[\sum_{\nu\in\mathcal{Z}_j(E)}2^{jd}(1+2^j|y|)^{-\frac{d-1}{2}}(1+2^j||y|-|t_\nu||)^{-N}\]
\[\lesssim\sup_{\nu\in\mathcal{Z}_j(E)}2^{jd}(1+2^j|y|)^{-\frac{d-1}{2}}(1+2^j||y|-|t_\nu||)^{-N}\]
\[\lesssim 1.\]
The bound follows upon integration.
\end{proof}
\begin{proof}[Proof of $(1,1,1)$ estimate \eqref{eq:core111}]
By Young's inequality, we see that
\[\left\|\mathcal{A}_{j}f(x,t_\nu)\right\|_{L^1_x(\RR^d;\ell^1_\nu(\mathcal{Z}_j(E)))}\]
\[\lesssim_N 2^{-\frac{j(d-1)}{2}}\sum_{\nu\in\mathcal{Z}_j(E)}\int\int \left|2^{jd}(1+2^j|y|)^{-\frac{d-1}{2}}(1+2^j||y|-|t_\nu||)^{-N} f(x-y)\right|dydx\]
\[\lesssim 2^{-\frac{j(d-1)}{2}}|\mathcal{Z}_j(E)|\sup_{\nu\in\mathcal{Z}_j(E)}\int \left|2^{jd}(1+2^j|y|)^{-\frac{d-1}{2}}(1+2^j||y|-|t_\nu||)^{-N} \right|dy\|f\|_{L^1}\]
\[\lesssim |\mathcal{Z}_j(E)|\|f\|_{L^1}.\]
\end{proof}
\begin{proof}[Proof of $(0,0,0)$ estimate \eqref{eq:coreinfinfinf}]
Again, by Young's inequality, we find that, for each $\nu$
\[\left|\mathcal{A}_{j}f(x,t_\nu)\right|\]
\[\lesssim_N 2^{-\frac{j(d-1)}{2}}\int \left|2^{jd}(1+2^j|y|)^{-\frac{d-1}{2}}(1+2^j||y|-|t_\nu||)^{-N} \right|dy\|f\|_{L^\infty}\]
\[\lesssim \|f\|_{L^\infty},\]
where the bound can be taken uniformly in $\nu$ and $x$.
\end{proof}

\section{Interpolation}\label{sec:interpolation}
\subsection{Localisation and $\ell^r$ embedding}\label{sec:localisation}
Besides interpolation of inequalities, we have a number of localisation results that allow us to extend the range of $(1/p,1/q,1/r)$ estimates from a given $(1/p_0,1/q_0,1/r_0)$. The first is obtained simply by a direct application of H\"older's inequality in the $t$-variable.
\begin{lemma}\label{lem:holderTime}Suppose that $0< r \leq r_0$. Then we have that 
\[\left\|\mathcal{A}_{j}f(x,t_\nu)\right\|_{L^q_x(\RR^d;\ell^r_\nu(\mathcal{Z}_j(E)))}\lesssim \left|\mathcal{Z}_j(E)\right|^{\frac{1}{r}-\frac{1}{r_0}}\left\|\mathcal{A}_{j}f(x,t_\nu)\right\|_{L^q_x(\RR^d;\ell^{r_0}_\nu(\mathcal{Z}_j(E)))}.\]
\end{lemma}
\begin{remark}Note that the lemma is stated for all $r>0$: this embedding may still be utilised in the analysis of variation seminorm estimates for generalisations of variation seminorm spaces to exponents $r<1$.
\end{remark}

\begin{lemma}\label{lem:holderSpace}Suppose that $1\leq r \leq \infty$ and $p_0\leq p \leq q \leq q_0$, with $(p_0,q_0,r)$ such that
\[\left\|\mathcal{A}_jf(x,t_\nu)\right\|_{L^{q_0}_x(\RR^d;\ell^r_\nu(\mathcal{Z}_j(E)))}\lesssim C_{j,p_0,q_0}\left\|f\right\|_{L^{p_0}(\RR^d)} .\]
Then 
\[\left\|\mathcal{A}_jf(x,t_\nu)\right\|_{L^q_x(\RR^d;\ell^r_\nu(\mathcal{Z}_j(E)))}\lesssim_N C_{j,p_0,q_0}\left\|f\right\|_{L^p(\RR^d)} + 2^{-jN} \left\|f\right\|_{L^p(\RR^d)}.\]
\end{lemma}
\begin{proof}
Recall, with \[\mathcal{A}_{j}f(x,t)=2^{-\frac{j(d-1)}{2}}(2\pi)^{-d-1}\sum_{\pm}\int \kappa_{j}^{\pm}(y,t) f(x-y)dy,\]
the kernel is bounded 
\[|\kappa_{j}^{\pm}(x,t)|\lesssim_{N} 2^{jd}(1+2^j|x|)^{-\frac{d-1}{2}}(1+2^j||x|-|t||)^{-N}.\]
For $|x|\gg 1$ and $|t|\sim 1$, we have that 
\[|\kappa_{j}^{\pm}(x,t)|\lesssim_{N} 2^{-jM}(1+|x|)^{-N},\]
so we can decompose 
\[\kappa_{j}^{\pm}(x,t)=\tilde{\kappa}_{j}^{\pm}(x,t)+E_j(x,t),\]
where $ |E_j(x,t)|\lesssim 2^{-jM}(1+|x|)^{-M}$ and $\tilde{\kappa}_j$ is supported in the region $|x|\lesssim 1$.
Thus we find that 
\[\left(\int\left(\sum_{\nu\in\mathcal{Z}_j(E)}\left|\mathcal{A}_{j}f(x,t_\nu)\right|^r\right)^{\frac{q}{r}}dx\right)^{\frac{1}{q}}\]
\[\lesssim \left(\int_{|x|\lesssim 1}\left(\sum_{\nu\in\mathcal{Z}_j(E)}\left|\mathcal{A}_{j}f(x,t_\nu)\right|^r\right)^{\frac{q}{r}}dx\right)^{\frac{1}{q}}\]
\[+ 2^{-\frac{j(d-1)}{2}}\left(\int_{|x|\gg 1}\left(\sum_{\nu\in\mathcal{Z}_j(E)}\left|\int E_j(y,t_\nu)f(x-y,t_\nu)dy\right|^r\right)^{\frac{q}{r}}dx\right)^{\frac{1}{q}}.\]
We can bound the first expression by H\"older's inequality, the assumed inequality, followed by a further application of H\"older's inequality. To bound the second term, we choose $s$ such that 
\[1+\frac{1}{q}=\frac{1}{p}+\frac{1}{s}\]
and observe that we can ensure by a suitably large choice of $M$ we can bound it by 
\[2^{-jN}\|f\|_{L^p}\]
by an application of our bound on $\sup_{t\in[1,2]}|E_j(\cdot,t)|$ and Young's inequality.
\end{proof}

The next result follows simply from the embedding of $\ell^r$ norms.
\begin{lemma}\label{lem:ellrEmbedding}Suppose that $ r \geq r_0$. Then we have that 
\[\left\|\mathcal{A}_{j}f(x,t_\nu)\right\|_{L^q_x(\RR^d;\ell^r_\nu(\mathcal{Z}_j(E)))}\lesssim \left\|\mathcal{A}_{j}f(x,t_\nu)\right\|_{L^q_x(\RR^d;\ell^{r_0}_\nu(\mathcal{Z}_j(E)))}.\]
\end{lemma}

\subsection{Variation seminorm bounds}\label{subsec:endpointInterpolation}
In this section we carry out the proof of our main results. These are obtained simply by combining the estimates for $\mathcal{A}_j$ in Sections \ref{subsec:oscIntEsts} and \ref{subsec:kernelEsts} with linear interpolation and the tools introduced in Section \ref{sec:localisation}.

\begin{proof}[Proof of Theorem \ref{thm:main}]
It suffices to show that, for 
 \[\left(\frac{1}{p},\frac{1}{q},\frac{1}{r}\right)\in\left(\conv \left\lbrace Q_1,  \widetilde{Q_{2,\beta}}, \widetilde{Q_{3,\beta}}, \widetilde{Q_{4,\gamma}},Q_{2,\beta}, Q_{3,\beta}, Q_{4,\beta}, Q_{D,\beta,\gamma}, Q_{B,\beta}, Q_{C,\beta},Q_{D,\beta,\gamma}^1\right \rbrace\right)^\circ,\]
there exists some $\epsilon>0$ such that 
\[\|V^r\left(\mathcal{A}_jf(x,\cdot)\right)\|_{L^q_x}\lesssim 2^{-j\epsilon}\|f\|_{L^p}.\]

We prioritise the analysis for the estimates that can be obtained close to key faces of the above polytope. These faces are
\[\conv \left\lbrace Q_1,   Q_{D,\beta,\gamma},Q_{4,\gamma}\right \rbrace\]
\[\conv \left\lbrace Q_{B,\beta},Q_{2,\beta}, Q_{3,\beta},Q_{C,\beta}\right \rbrace\]
\[\conv \left\lbrace  Q_{C,\beta},Q_{3,\beta},Q_{4,\gamma}, Q_{D,\beta,\gamma}\right \rbrace.\]
Essentially, these faces correspond to critical thresholds for the interpolation of particular estimates: $(0,0,0)$--$(1/2,1/q_\gamma,1/2)$--$(1,0,1)$,  $(1/2,1/2,1/2)$--$(1,1,1)$--$(1,0,1)$ and $(1/2,1/2,1/2)$--$(1,1/q_\gamma,1)$--$(1,0,1)$, respectively. The fact the critical threshold is of two dimensions rather than a 1 dimensional subset of a planar segment reflects the fact we can also utilise the localisation property (H\"older's inequality) in the $\ell^r$ norm. The remaining estimates can be obtained once we have properly analysed the above faces by a simple application of the tools of Section \ref{sec:localisation}.

We consider estimates close to the first face \[\conv \left\lbrace Q_1,   Q_{D,\beta,\gamma},Q_{4,\gamma}\right \rbrace.\]
Corresponding to $Q_1=(0,0,0)$, we have the estimate \eqref{eq:coreinfinfinf}
\[\|\mathcal{A}_jf(x,t)\|_{L^\infty_x(\ell^\infty_t)}\lesssim \|f\|_{L^\infty}.\]

To see what happens close to $Q_{4,\gamma}$ and  $Q_{D,\beta,\gamma}$, we interpolate between the estimates $(1,0,1)$ and $(1/2,1/q_\gamma,1/2)$, \eqref{eq:core1inf1} and \eqref{eq:core2q2}. At $(1/p_\theta,1/q_\theta,1/r_\theta)=\theta(1,0,1)+(1-\theta)(1/2,1/q_\gamma,1/2)$, we have the estimate
\begin{equation}\label{eq:1012q2}\|\mathcal{A}_jf(x,t_\nu)\|_{L_x^{q_\theta}(\ell_\nu^{r_\theta}(\mathcal{Z}_j(E))}\lessapprox 2^{j\theta}2^{-\frac{(1-\theta)((d-1)^2-2\gamma)}{2(d-1+2\gamma)}}\|f\|_{L^{p_\theta}}.\end{equation}
Note that the $(1/2,1/q_\gamma,1/2)$ estimate is summable for all $d\geq 3$, but we require $\gamma\leq 1/2 \leq (-3+\sqrt{17})/4$ in the case $d=2$ (which is fine because $ \gamma\leq (-3+\sqrt{17})/4\leq 1/2$). It is easier to first consider $Q_{4,\gamma}$, the threshold at which the estimate \eqref{eq:1012q2} fails to be summable. The critical $\theta$ is such that $\theta-(1-\theta)\left(\frac{(d-1)^2-2\gamma}{2(d-1+2\gamma)}\right)=0$. Such a $\theta$ is given by
\[\theta\left(1+\frac{(d-1)^2-2\gamma}{2(d-1+2\gamma)}\right)=\frac{(d-1)^2-2\gamma}{2(d-1+2\gamma)}\]
\[\theta=\frac{(d-1)^2-2\gamma}{\left(2(d-1+2\gamma)+(d-1)^2-2\gamma\right)}=1-\frac{2(d-1+2\gamma)}{\left(2(d-1+2\gamma)+(d-1)^2-2\gamma\right)}.\]
The corresponding $(p,q,r)$ are such that 
\[\frac{1}{p}=\frac{1}{r}=\frac{(d-1)^2-2\gamma}{\left(2(d-1+2\gamma)+(d-1)^2-2\gamma\right)}+\frac{d-1+2\gamma}{\left(2(d-1+2\gamma)+(d-1)^2-2\gamma\right)}\]
\[=\frac{d(d-1)}{(d^2-1+2\gamma)}\]
and
\[\frac{1}{q}=\frac{2(d-1+2\gamma)}{\left(2(d-1+2\beta)+(d-1)^2-2\beta\right)}\frac{1}{q_{\beta}}.\]
\[q_\gamma=\frac{2(d-1+2\gamma)}{d-1}\]
So we get
\[\frac{1}{q}=\frac{d-1}{\left(d^2-1+2\gamma\right)}.\]
The critical vertex is thus $Q_{4,\gamma}$.

If we apply Lemma \ref{lem:holderTime} to \eqref{eq:1012q2}, we obtain, for $r\leq r_\theta$,
\begin{equation}\label{eq:1012q2Holder}\|\mathcal{A}_jf(x,t_\nu)\|_{L^{q_\theta}_x(\ell^r_\nu(\mathcal{Z}_j(E)))}\lessapprox |\mathcal{Z}_j(E)|^{\frac{1}{r}-\frac{1}{r_\theta}}2^{j\theta}2^{-\frac{(1-\theta)((d-1)^2-2\gamma)}{2(d-1+2\gamma)}}\|f\|_{L^{p_\theta}},\end{equation}
where $(1/p_\theta,1/q_\theta,1/r_\theta)=\theta(1,0,1)+(1-\theta)(1/2,1/q_\gamma,1/2)$. We can ensure this estimate is summable when $r=1$ and $\theta=0$ because $\gamma\leq -1+\sqrt{3}$ in the case $d=3$ or $\gamma \leq (-3+\sqrt{17})/4$ in the case $d=2$. Indeed, for $r\leq 2$,
\[\left\|\mathcal{A}_{j}f(x,t_\nu)\right\|_{L^{q_\gamma}_x(\RR^d;\ell^r_\nu(\mathcal{Z}_j(E)))}\lesssim |\mathcal{Z}_j(E)|^{\frac{1}{r}-\frac{1}{2}}2^{-j\frac{(d-1)^2-2\gamma}{2(d-1+2\gamma)}}\left\|f\right\|_{L^2}.\]
With $r=1$, we have 
$2^{\frac{j}{2}}2^{-j\frac{(d-1)^2-2\gamma}{2(d-1+2\gamma)}}$.
In the case $d=3$ this is
$2^{\frac{j\beta}{2}}2^{-j\frac{4-2\gamma}{2(2+2\gamma)}}$, which is summable only for $\gamma$ such that
\[(1+\gamma)\beta\leq 2-\gamma.\]
More restrictive is the inequality $(1+\gamma)\gamma\leq 2-\gamma$, which holds provided $\gamma\leq -1+\sqrt{3}$. In the case $d=2$ it is
$2^{\frac{j\beta}{2}}2^{-j\frac{1-2\gamma}{2(1+2\gamma)}}$, which is summable only for $\gamma$ such that
\[(2+2\gamma)\beta\leq 1-2\gamma.\]
More restrictive is the inequality $(1+2\gamma)\gamma\leq 1-2\gamma$, which holds provided $\gamma\leq (-3+\sqrt{17})/4$. 
The vertex $Q_{D,\beta,\gamma}$ corresponds to the critical exponent at which we can take $r=1$ and the estimate fails to be summable. The critical $\theta$ is such that 
\[\beta\left(1-\frac{1}{r_\theta}\right) +\theta-(1-\theta)\left(\frac{(d-1)^2-2\gamma}{2(d-1+2\gamma)}\right)=0.\]
The solution is given by 
\[\theta = \frac{2(d-1+2\gamma)}{d^2-\beta d + \beta + 2\gamma - 2 \beta\gamma -1}.\]
The corresponding $(p,q,r)$ are such that 
\[\left(\frac{1}{p},\frac{1}{q},\frac{1}{r}\right)=\left(\frac{2-\theta}{2},\theta\left(\frac{d-1}{2(d-1+2\gamma)}\right),1\right).\]
This vertex is precisely $Q_{D,\beta,\gamma}$.

Interpolating the estimates at $Q_{1}$, (close to) $Q_{4,\gamma}$, and (close to) $Q_{D,\beta,\gamma}$ shows there exist points $(1/p,1/q,1/r)$ arbitrarily close to any point of \[\conv \left\lbrace Q_1,   Q_{D,\beta,\gamma},Q_{4,\gamma}\right \rbrace\]
for which there exists $\epsilon>0$ such that 
\[\|\mathcal{A}_jf(x,t)\|_{L^q(\ell^r_t)}\lesssim 2^{-j\epsilon}\|f\|_{L^p}.\]

Let's now consider the face \[\conv \left\lbrace Q_{B,\beta},Q_{2,\beta}, Q_{3,\beta},Q_{C,\beta}\right \rbrace.\]

To see what happens close to $Q_{B,\beta}$ and $Q_{2,\beta}$, we interpolate between the $(1,1,1)$ estimate and the $(\frac{1}{2},\frac{1}{2},\frac{1}{2})$ estimates \eqref{eq:core111} and \eqref{eq:core222}. With $(1/p_\theta,1/q_\theta,1/r_\theta)=\theta(1,1,1)+(1-\theta)(1/2,1/2,1/2)$, we have the estimate
\begin{equation}\label{eq:111222}
\|\mathcal{A}_jf(x,t)\|_{L^{q_\theta}(\ell^{r_\theta})}\lesssim  |\mathcal{Z}_j(E)|^{\theta}2^{-(1-\theta)\frac{(d-1)}{2}} |\mathcal{Z}_j(E)|^{\frac{(1-\theta)}{2}}\|f\|_{L^{p_\theta}}.
\end{equation}
The vertex $Q_{2,\beta}$ corresponds to the threshold at which \eqref{eq:111222} fails to be summable. The critical $\theta$ is such that $\theta\beta-(1-\theta)\left(\frac{(d-1)}{2}-\frac{\beta}{2}\right)=0$. Such a $\theta$ is given by
\[\theta=\frac{d-1-\beta}{\beta+ d-1}=1-\frac{2\beta}{\beta+d-1}\]
and the corresponding $(p,q,r)$ are such that
\[\frac{1}{p}=\frac{1}{q}=\frac{1}{r}=\frac{d-1-\beta}{\beta+ d-1}+\frac{\beta}{\beta+d-1}=\frac{d-1}{\beta+d-1}.\]
The critical vertex is thus $Q_{2,\beta}$. If we apply Lemma \ref{lem:holderTime} to \eqref{eq:111222}, we obtain, for $r\leq r_\theta$,
\[\|\mathcal{A}_jf(x,t_\nu)\|_{L^{q_\theta}(\ell_\nu^{r_\theta}(\mathcal{Z}_j(E))}\lesssim |\mathcal{Z}_j(E)|^{\frac{1}{r}-\frac{1}{r_\theta}}|\mathcal{Z}_j(E)|^{\theta}2^{-j(1-\theta)\left(\frac{(d-1)}{2}\right)} |\mathcal{Z}_j(E)|^{\frac{(1-\theta)}{2}}\|f\|_{L^{p_\theta}},\]
where $(1/p_\theta,1/q_\theta,1/r_\theta)=\theta(1,1,1)+(1-\theta)(1/2,1/2,1/2)$. The vertex $Q_{B,\beta}$ corresponds to the threshold $\theta$ at which this estimate fails to be summable with $r=1$. Noting that $1/r_\theta=(1+\theta)/2$, we see that the critical $\theta$ satisfies the equation
\[\beta\left(1-(1+\theta)/2\right)+\beta\theta-(1-\theta)\frac{(d-1)}{2}+\beta(1-\theta)/2=0.\]
Thus we find $\theta$ is given by
\[\theta=\frac{d-1-2\beta}{d-1}=1-\frac{2\beta}{d-1}.\]
The corresponding vertex is $Q_{B,\beta}$. 

To see what happens close to $Q_{3,\beta}$ and $Q_{C,\beta}$, we interpolate the $(1/2,1/2,1/2)$ and $(1,0,1)$ estimates \eqref{eq:core222} and \eqref{eq:core1inf1}. For $(1/p_\theta,1/q_\theta,1/r_\theta)=(1-\theta)(1/2,1/2,1/2)+\theta(1,0,1)$, we have the estimate
\begin{equation}\label{eq:2221inf1}\left\|\mathcal{A}_jf(x,t_\nu)\right\|_{L^{q_\theta}_x(\RR^d;\ell^{r_\theta}_\nu(\mathcal{Z}_j(E))}\lesssim 2^{j\theta}2^{-j(1-\theta)\frac{(d-1)}{2}}|\mathcal{Z}_j(E)|^{(1-\theta)/2}\|f\|_{L^{p_\theta}}.\end{equation}
The critical exponent $Q_{3,\beta}$ is obtained when we find $\theta$ such that $\theta-(1-\theta)\left(\frac{(d-1)}{2}-\frac{\beta}{2}\right)=0$. Such a $\theta$ is given by
\[\theta=\frac{d-1-\beta}{d+1-\beta}=1-\frac{2}{d+1-\beta}.\]
The corresponding $(p,q,r)$ are such that 
\[\frac{1}{p}=\frac{1}{r}=\frac{d-1-\beta}{d+1-\beta}+\frac{1}{d+1-\beta}=\frac{d-\beta}{d+1-\beta}\]
and
\[\frac{1}{q}=\frac{1}{d+1-\beta}.\]
The critical vertex is thus $Q_{3,\beta}$. 
We can instead apply Lemma \ref{lem:holderTime} to \eqref{eq:2221inf1}. Notice that $1/r_\theta=(1-\theta)/2+\theta=(1+\theta)/2$. At $(1/p_\theta,1/q_\theta,1)=(1-\theta)(1/2,1/2,1)+\theta(1,0,1)$, we have the estimate
\[\left\|\mathcal{A}_jf(x,t_\nu)\right\|_{L^{q_\theta}_x(\RR^d;\ell^{1}_\nu(\mathcal{Z}_j(E))}\lesssim |\mathcal{Z}_j(E)|^{1-(1+\theta)/2}2^{j\theta}2^{-j(1-\theta)\frac{(d-1)}{2}}|\mathcal{Z}_j(E)|^{(1-\theta)/2}\|f\|_{L^{p_\theta}}.\]
The critical $\theta$ is now such that
\[(1-\theta)\beta/2+\theta-(1-\theta)(d-1)/2+(1-\theta)\beta/2=0.\]
Such a $\theta$ is given by
\[ \theta=\frac{d-1-2\beta}{d+1-2\beta}=1-\frac{2}{d+1-2\beta}.\]
The corresponding vertex is $Q_{C,\beta}$.

A further interpolation of the estimates we obtained close to the critical thresholds shows that arbitrarily close to any point of the face $\left\lbrace Q_{B,\beta},Q_{2,\beta}, Q_{3,\beta},Q_{C,\beta}\right \rbrace$ there exists $(1/p,1/q,1/r)$ such that
\[\left\|\mathcal{A}_jf(x,t_\nu)\right\|_{L^{q}_x(\RR^d;\ell^{r}_\nu(\mathcal{Z}_j(E))}\lesssim 2^{-j\epsilon}\|f\|_{L^p},\]
for some $\epsilon>0$.

We now consider the face
\[\conv \left\lbrace  Q_{C,\beta},Q_{3,\beta},Q_{4,\gamma}, Q_{D,\beta,\gamma}\right \rbrace.\]
The relevant estimates are the $(1/2,1/2,1/2)$, $(1/2,1/q_\gamma,1/2)$, and $(1,0,1)$ estimates \eqref{eq:core222}, \eqref{eq:core2q2}, and \eqref{eq:core1inf1}. In fact, we already considered what happens close to the line connecting $Q_{C,\beta}$ and $Q_{3,\beta}$ in our analysis of the face $\conv \left\lbrace Q_{B,\beta},Q_{2,\beta}, Q_{3,\beta},Q_{C,\beta}\right \rbrace$. Likewise, we already considered what happens near the line between  $Q_{4,\gamma}$ and $ Q_{D,\beta,\gamma}$ in our analysis of the face 
\[\conv \left\lbrace Q_1,   Q_{D,\beta,\gamma},Q_{4,\gamma}\right \rbrace.\]
Interpolating the estimates as previously, there exists some $(1/p,1/q,1/r)$ arbitrarily close to any point of $\conv \left\lbrace  Q_{C,\beta},Q_{3,\beta},Q_{4,\gamma}, Q_{D,\beta,\gamma}\right \rbrace$ for which
\[\left\|\mathcal{A}_jf(x,t_\nu)\right\|_{L^{q}_x(\RR^d;\ell^{r}_\nu(\mathcal{Z}_j(E))}\lesssim 2^{-j\epsilon}\|f\|_{L^p},\]
for some $\epsilon>0$.

The proof is now completed by applying Lemma \ref{lem:holderSpace} and Lemma \ref{lem:ellrEmbedding}.
\end{proof}


We now turn to Theorems \ref{thm:LSimplies} and Theorem \ref{thm:d2LSimplies}. These concern the variation seminorm estimates that can be obtained subject to a conjectured local smoothing estimate \eqref{eq:coreqqq}. To recall, this estimate is expressed as 
\[\left\|\mathcal{A}_{j}f(x,t_\nu+h)\right\|_{L^{q_{\mathrm{LS},\beta}}_x(\RR^d;\ell^{q_{\mathrm{LS},\beta}}_\nu(\mathcal{Z}_j(E)))}\lessapprox 2^{-\frac{j(d-1)}{2}}|\mathcal{Z}_j(E)|^{1/ q_{\mathrm{LS},\beta}}\|f\|_{L^{q_{\mathrm{LS},\beta}}}.\]
Many of the interpolated estimates utilised in the proof of Theorem \ref{thm:main} remain valid, so the following proofs are correspondingly abridged.

\begin{proof}[Proof of Theorem \ref{thm:LSimplies}]
Applying Lemma \ref{lem:holderTime} to the local smoothing estimate \eqref{eq:coreqqq} we obtain, for $r\leq q_{\mathrm{LS},\beta}$,
\[\left\|\mathcal{A}_{j}f(x,t_\nu+h)\right\|_{L^{q_{\mathrm{LS},\beta}}_x(\RR^d;\ell^{r}_\nu(\mathcal{Z}_j(E)))}\lessapprox 2^{-\frac{j(d-1)}{2}}|\mathcal{Z}_j(E)|^{1/r}\|f\|_{L^{q_{\mathrm{LS},\beta}}}.\]

Observe that
\[-\frac{(d-1)}{2}+\beta \leq 0.\]
We interpolate with the $(0,0,0)$ estimate. For $\theta\in (0,1)$ and $(1/p_\theta,1/q_\theta,1/r_\theta)=\theta(1/q_{\mathrm{LS},\beta},1/q_{\mathrm{LS},\beta},1/q_{\mathrm{LS},\beta})$ and $r\leq r_\theta$,
\[\left\|\mathcal{A}_{j}f(x,t_\nu+h)\right\|_{L^{q_{\theta}}_x(\RR^d;\ell^{r}_\nu(\mathcal{Z}_j(E)))}\lessapprox 2^{-\frac{j\theta(d-1)}{2}}|\mathcal{Z}_j(E)|^{1/r}\|f\|_{L^{q_{\theta}}}.\]
The critical $\theta$ when $r=1$ satisfies
\[-\frac{\theta(d-1)}{2}+\beta  = 0.\]
This corresponds to 
\[Q_{A,\beta}=\left(\frac{\beta}{(d-1+\beta)},\frac{\beta}{(d-1+\beta)},1\right).\]

It essentially remains to adjust our analysis of estimates close to the vertex $Q_{D,\beta,\gamma}$. Here we interpolate between the $(1,0,1)$ and local smoothing estimate at $(1/q_{\mathrm{LS},\beta},1/q_{\mathrm{LS},\beta},1/q_{\mathrm{LS},\beta})$ and . With $(1/p_\theta,1/q_\theta,1/r_\theta)=(1-\theta)(1,0,1)+\theta(1/q_{\mathrm{LS},\beta},1/q_{\mathrm{LS},\beta},1/q_{\mathrm{LS},\beta})$ and for $r\leq r_{\theta}$, we have the estimate
\[\left\|\mathcal{A}_{j}f(x,t_\nu+h)\right\|_{L^{q_{\theta}}_x(\RR^d;\ell^{r}_\nu(\mathcal{Z}_j(E)))}\lessapprox |\mathcal{Z}_j(E)|^{1/r-\frac{1}{r_\theta}} 2^{j(1-\theta)}2^{-j\theta\frac{(d-1)}{2}}|\mathcal{Z}_j(E)|^{\theta\frac{(d-1)}{2(d-1+\beta)}}\|f\|_{L^{p_{\theta}}}.\]
When $r=1$ the critical threshold at which this fails to be summable is determined by the equation
\[\theta\beta\left(1-\frac{(d-1)}{2(d-1+\beta)}\right)+(1-\theta)-\theta\frac{(d-1)}{2}+\beta\theta\frac{(d-1)}{2(d-1+\beta)}=0.\]
Equivalently,
\[\theta=\frac{2(d-1+\beta)}{(d-1)^2+(2-\beta)(d-1)+2\beta(1-\beta) }=1-\frac{(d-1)^2-\beta(d-1)-2\beta^2}{(d-1)^2+(2-\beta)(d-1)+2\beta(1-\beta) }\]
The corresponding $(1-\theta)(1,0,1)+\theta(1/q_{\mathrm{LS},\beta},1/q_{\mathrm{LS},\beta},1/q_{\mathrm{LS},\beta})$ is given by $Q_{D,\beta,\beta}$.
\end{proof}

\begin{proof}[Proof of Theorem \ref{thm:d2LSimplies}]
Applying Lemma \ref{lem:holderTime} to the local smoothing estimate \eqref{eq:coreqqq} we obtain, for $r\leq q_{\mathrm{LS},\beta}$,
\[\left\|\mathcal{A}_{j}f(x,t_\nu+h)\right\|_{L^{q_{\mathrm{LS},\beta}}_x(\RR^d;\ell^{r}_\nu(\mathcal{Z}_j(E)))}\lessapprox 2^{-\frac{j(d-1)}{2}}|\mathcal{Z}_j(E)|^{1/r}\|f\|_{L^{q_{\mathrm{LS},\beta}}}.\]

Likewise, applying Lemma \ref{lem:holderTime} to the $(1/2,1/2,1/2)$ \eqref{eq:core222} we obtain, for $r\leq 2$,
\[\left\|\mathcal{A}_{j}f(x,t_\nu+h)\right\|_{L^{2}_x(\RR^d;\ell^{r}_\nu(\mathcal{Z}_j(E)))}\lessapprox 2^{-\frac{j(d-1)}{2}}|\mathcal{Z}_j(E)|^{1/r}\|f\|_{L^{2}}.\]

The critical $r$ at which these estimates fail to be summable is 
\[\frac{1}{r}=\frac{d-1}{2\beta}.\]
The corresponding critical exponents are indexed by
\[Q_{5,\beta,\beta}=\left(\frac{d-1}{2(d-1+\beta)},\frac{d-1}{2(d-1+\beta)},\frac{d-1}{2\beta}\right)\]
and
\[Q_{6,\beta}=\left(\frac{1}{2},\frac{1}{2},\frac{d-1}{2\beta}\right).\]

Recall that in the case $\gamma>1/2$, the estimate \eqref{eq:core2q2} was not summable. With reference to our previous proof of Theorem \ref{thm:main}, the remaining modifications to consider are in our analysis of estimates close to $Q_{4,\gamma}$.

At $(1/p_\theta,1/q_\theta,1/r_\theta)=(1-\theta)(1,0,1)+\theta(1/q_{\mathrm{LS},\beta},1/q_{\mathrm{LS},\beta},1/q_{\mathrm{LS},\beta})$, we have the estimate
\[\left\|\mathcal{A}_{j}f(x,t_\nu+h)\right\|_{L^{q_{\theta}}_x(\RR^d;\ell^{r_{\theta}}_\nu(\mathcal{Z}_j(E)))}\lessapprox  2^{j(1-\theta)}2^{-j\theta\frac{(d-1)}{2}}|\mathcal{Z}_j(E)|^{\theta\frac{(d-1)}{2(d-1+\beta)}}\|f\|_{L^{p_{\theta}}}.\]
The critical $\theta$ at which this fails to be summable is determined by the equation
\[(1-\theta)-\theta\frac{(d-1)}{2}+\beta\theta\frac{(d-1)}{2(d-1+\beta)}=0.\]
Equivalently,
\[\theta=\frac{2(d-1+\beta)}{d^2-1+2\beta}=1-\frac{d^2-2d+1}{d^2-1+2\beta}.\]
The corresponding $(1-\theta)(1,0,1)+\theta(1/q_{\mathrm{LS},\beta},1/q_{\mathrm{LS},\beta},1/q_{\mathrm{LS},\beta})$ is $Q_{4,\gamma}$.
\end{proof}

\section{Sharpness}\label{sec:sharp}
\subsection{}\label{sec:translation}
The condition $p\leq q$ is a classic necessary condition for translation invariant operators \cite{hormander60}. 
The vertices $Q_{A,\beta,\beta}, Q_{B,\beta},Q_1,Q_{2,\beta},\widetilde{Q_{2,\beta}},Q_{5,\beta,\beta},Q_{6,\beta}$ all satisfy this with equality.

\subsection{}\label{sec:Shell}
Let $f_{j}(x)=\chi_{[t_1-2^{-j},t_1+2^{-j}]}(|x|)$, for some $t_1\in E$. It is simple to verify that for $y\in B(0,2^{-j-4})$ there exists $t_1\in E$ such that $|\sigma_{t_1}*f_{j}(y)|\sim 1$. Also, there is $t_2\in E$ such that, for the same $y$, $\sigma_{t_2}*f_{j}(y)=0$. We can also see that 
\[ \|f_{j}\|_{p}\lesssim 2^{-\frac{j}{p}}. \]
Testing the inequality 
\[\|V^r_E\left(\mathcal{A}f_{j}(x,\cdot)\right)\|_{L^q_x}\lesssim \|f_{j}\|_{p}\]
as $j\rightarrow \infty$, we require
\[2^{-\frac{jd}{q}}\lesssim 2^{-\frac{j}{p}},\]
from which we obtain the necessary condition
\[\frac{1}{p}\leq \frac{d}{q}.\]

The vertices $(1/p,1/q,1/r)\in \lbrace Q_1, Q_{4,\gamma}, \widetilde{Q_{4,\gamma}}\rbrace$ satisfy this inequality with equality. This test corresponds with the yellow faces in Figures \ref{fig:thmmain}, \ref{fig:thmLS}, and \ref{fig:d2LS}. 

\subsection{}\label{sec:Ball}
Let $f_{j}(x)= \chi_{B(0,2^{-j})}(x)$. We can establish that, for $|y|\in N_{2^{-j-3}}(E)$,
\[V^r_E\left(\mathcal{A}f_{j}(y,\cdot)\right)\gtrsim 2^{-j(d-1)}.\]
We also see that 
\[ \|f_{j}\|_{p}\lesssim 2^{-\frac{jd}{p}}.\]
 Testing the inequality 
\[\|V^r_E\left(\mathcal{A}f_{j}(y,\cdot)\right)\|_{L^q_y}\lesssim \|f_{j}\|_{p},\]
we require, for all $\epsilon>0$ and suitable $j$ which we can take arbitrarily large, 
\[2^{\frac{j(\beta-\epsilon)}{q}}2^{-\frac{j}{q}}2^{-j(d-1)}\lesssim 2^{-\frac{jd}{p}}.\]
Taking $j\rightarrow \infty$, we obtain the necessary condition
\[\frac{\beta-1}{q}+\frac{d}{p}\leq (d-1).\]

The vertices $(1/p,1/q,1/r)\in \lbrace \widetilde{Q_{3,\beta}},Q_{3,\beta},Q_{2,\beta},  \widetilde{Q_{2,\beta}}\rbrace$ satisfy this inequality with equality. This test corresponds with the blue faces in Figures \ref{fig:thmmain}, \ref{fig:thmLS}, and \ref{fig:d2LS}. 

\subsection{}\label{sec:Knapp}
Let $f_{j}(x)=\chi_{[-2^{-\frac{j}{2}},2^{-\frac{j}{2}}]^{d-1}\times[-2^{-j},2^{-j}]}(x)$. We can establish that for 
 $y\in R = [-2^{-\frac{j}{2}},2^{-\frac{j}{2}}]^{d-1}\times N_{2^{-j-4}(E)}$, 
\[V_r(\mathcal{A}f_{j})(y)\gtrsim 2^{-\frac{j(d-1)}{2}}.\]
We also see that 
\[ \|f_{j}\|_{p}\lesssim 2^{-\frac{j(d+1)}{2p}}. \]
Testing the inequality 
\[\| V^r_E\left(\mathcal{A}f_{j}(y,\cdot)\right)\|_{L^q_y}\lesssim \|f_{j}\|_{p}\]
as $j\rightarrow \infty$, we require
\[2^{-\frac{j(d-1)}{2q}}2^{-\frac{j(1-\beta)}{q}}2^{-\frac{j(d-1)}{2}}\lessapprox  2^{-\frac{j(d+1)}{2p}},\]
from which we obtain the necessary condition
\[\frac{d+1}{2p}\leq \frac{d-1}{2q}+\frac{d-1}{2}.\]

The vertices $(1/p,1/q,1/r)\in \lbrace Q_{3,\beta},  \widetilde{Q_{3,\beta}}\rbrace$ satisfy this inequality with equality. In the case $\beta=\gamma$, the vertices $(1/p,1/q,1/r)\in \lbrace \widetilde{Q_{4,\gamma}},  Q_{4,\gamma}\rbrace$ satisfy this inequality with equality. This test corresponds with the pink faces in Figures \ref{fig:thmmain}, \ref{fig:thmLS}, and \ref{fig:d2LS}.

\subsection{}\label{sec:assouadKnapp}
The following Knapp example, adapted to the Assouad spectrum, is taken from \cite{andersonHughesRoosSeeger21}. 
Given $\theta\in[0,1]$ and $\epsilon>0$, we consider a suitable scale $\delta=2^{-j}$ (which can be chosen arbitrarily small) such that a minimal cover of any $E\cap I_\theta$ by $\delta$-intervals, where $I_\theta$ is a suitable $\delta^\theta$-interval, contains $\gtrsim\delta^{-\gamma_\theta(1-\theta)+\epsilon}$  intervals.

We restrict our set $E$ to the suitable subinterval $I_\theta\subset [1,2]$ as above of width $\delta^\theta$. For some $t_0\in I_\theta$, we take $f_\delta$ to be the characteristic function of the $\delta$-neighbourhood of $t_0$-spherical cap of dimensions $\left(\delta^{(1-\theta)/2}\right)^{d-1}$.
The resulting inequality is
\[\left[\left(\frac{\delta}{\delta^{\theta}}\right)^{-\gamma_\theta+\epsilon}\times \delta \times \delta^{(d-1)(1+\theta)/2}\right]^{\frac{1}{q}}\times \delta^{(d-1)(1-\theta)/2}\]
\[\lesssim \left[\delta \times \delta^{(d-1)(1-\theta)/2}\right]^{\frac{1}{p}}.\]
Thus we obtain the necessary condition
\[-\frac{(1-\theta)\gamma_\theta}{q} +(d-1)\frac{(1+\theta)}{2q}+\frac{1}{q}+(1-\theta)\frac{(d-1)}{2}\geq \frac{1}{p}+(d-1)\frac{(1-\theta)}{2p}.\]

In the Assouad-regular case $\gamma_\theta=\gamma$, where $\theta=1-\beta/\gamma$, the vertices $(1/p,1/q,1/r)\in \lbrace \widetilde{Q_{4,\gamma}},  Q_{4,\gamma},Q_{3,\beta},  \widetilde{Q_{3,\beta}}\rbrace$ satisfy this inequality with equality. This test still corresponds with the pink faces in Figures \ref{fig:thmmain}, \ref{fig:thmLS}, and \ref{fig:d2LS} (however see \cite{roosSeeger23} with regards to the non-regular case).

\subsection{}\label{sec:multiShell}
Let $f_{j,\omega}(x)=\sum_{l=1}^{N_j} (-1)^{r_l(\omega)} \chi_{I_l}(|x|)$, $I_1,I_2,\ldots , I_{N_j}$ are $2^{-j}$-intervals 
covering $E$, with $N_j$ minimal and $2^{-j}$-separated endpoints and $r_l(\omega)$ are randomised signs. By dimensional considerations, for suitable $j$, which we can take arbitrarily large, $N_j\geq 2^{j(\beta-\epsilon)}$ for any $\epsilon>0$. It is simple to verify  that for $t\in E$ and $y\in B(0,2^{-j-4})$ $|\sigma_t*f_{j,\omega}(y)|\sim 1$. By a randomisation argument
, since there are $N_j$ intervals with randomised signs, there is with positive probability some $\omega$ such that 
\[V_r(\mathcal{A}f_{j,\omega})(y)\gtrsim N_j^{\frac{1}{r}}\gtrsim 2^{\frac{j(\beta-\epsilon)}{r}},\]
 for $y\in B(0,2^{-j-4})$. We can also see that 
\[\|f_{j,\omega}\|_{p}\lesssim \left(2^{-j}2^{j(\beta+\epsilon)}\right)^{\frac{1}{p}}.\] Testing the inequality 
\[\|V_r(\mathcal{A}f_{j,\omega})\|_{q}\lesssim \|f_{j,\omega}\|_{p}\]
as $j\rightarrow \infty$, we require
\[2^{-\frac{jd}{q}}2^{\frac{j(\beta-\epsilon)}{r}}\lesssim \left(2^{-j}2^{j(\beta+\epsilon)}\right)^{\frac{1}{p}},\]
from which we obtain the necessary condition
\[\frac{1-\beta}{p}+\frac{\beta}{r}\leq \frac{d}{q} .\] 

The vertices $(1/p,1/q,1/r)\in \lbrace Q_1, Q_{A,\beta,\beta},Q_{D,\beta,\gamma},Q_{4,\gamma},Q_{5,\beta,\beta}\rbrace$ satisfy this inequality with equality. This test corresponds with the orange faces in Figures \ref{fig:thmmain}, \ref{fig:thmLS}, and \ref{fig:d2LS}. 

\subsection{}\label{sec:multiBall}
We have not been able to construct an example in this region, but include a rough description of what one might like like (similar to \cite{beltranOberlinRoncalSeegerStovall22}). We require an array of $\delta^{-\beta-\epsilon}$ characteristic functions of balls of radius $\delta$, which are equipped with some randomised sign. Furthermore, we require that on a region $R$ of measure $C^{d-1}\times \left(\delta\times \delta^{-\beta+\epsilon}\right)$ the dilations of the sphere centred on a point in $R$ by $t\in E$ intersect with $\delta^{-\beta+\epsilon}$ of those balls. The resulting necessary condition would then be
\[\frac{d-\beta}{p}+\frac{\beta}{r}\leq (d-1) + \frac{1-\beta}{q} .\]

The vertices $(1/p,1/q,1/r)\in \lbrace Q_{3,\beta}, Q_{C,\beta},Q_{B,\beta},Q_{2,\beta},Q_{6,\beta}\rbrace$ satisfy this inequality with equality. This inequality corresponds with the red faces in Figures \ref{fig:thmmain}, \ref{fig:thmLS}, and \ref{fig:d2LS}. 

\subsection{}\label{sec:multiKnapp}
In the following example we consider Cantor sets $E\subset [1,2]$ whose first generation $E_1$ can be expressed
\[E_1=\cup_{k=0}^{2^{a}-1}[1+k2^{-a},1+k2^{-a}+2^{-b}],\] 
for integers $2\leq a < b$. The Minkowski dimension of $E$ is $\beta=\frac{a}{b}$. This set has significant arithmetic structure: note that, for $1\leq l \leq 2^{a-1}$, $|E_1\cap (E_1 + l2^{-a})|\geq |E_1|/2$.

We let $\delta=2^{-jb}$ be some small positive quantity. Notice that the left endpoints of the intervals of $E_{j}$ (the $j$th generation of the Cantor set) are uniformly $\delta^{\beta}$-separated and there are $\delta^{-\beta}$ such intervals. 
Let $f_{j,\omega}(x)=\sum_{l=1}^{2^{ja}} (-1)^{r_l(\omega)} \chi_{[-\delta,\delta]\times[-\delta^{\frac{1}{2}},\delta^{\frac{1}{2}}]^{d-1}}(x-(1+l2^{-ja})e_1)$.
By a randomisation argument, we can establish that for some $\omega$ and suitable subset 
 $R \subset [-\delta^{\frac{1}{2}},\delta^{\frac{1}{2}}]^{d-1}$ with $|R|\gtrsim \delta^{\frac{d-1}{2}}|E_{j}|\sim \delta^{\frac{d-1}{2}} \delta^{1-\beta}$, $y\in R$ implies
\[V_r(\mathcal{A}f_{j,\omega_j})(y)\gtrsim \delta^{\frac{\beta}{r}}\delta^{\frac{(d-1)}{2}}.\]
We also have that $\|f_j\|_p\sim \delta^{-\frac{\beta}{p}}\delta^{\frac{d-1}{2p}}\delta^{\frac{1}{p}}$.
Testing the inequality 
\[\|V_r(\mathcal{A}f_{j,\omega_j})\|_{q}\lesssim \|f_{j,\omega_j}\|_{p}.\]
As $\delta\rightarrow 0$ we obtain the necessary condition
\[\frac{d-1}{2q}+\frac{1-\beta}{q}-\frac{\beta}{r}+\frac{d-1}{2}\geq -\frac{\beta}{p}+\frac{d-1}{2p}+\frac{1}{p}.\]

The vertices $(1/p,1/q,1/r)\in \lbrace Q_{C,\beta},Q_{3,\beta}\rbrace$ satisfy this inequality with equality. In the case $\beta=\gamma$, the vertices $(1/p,1/q,1/r)\in \lbrace Q_{4,\gamma},Q_{D,\beta,\gamma},Q_{5,\beta,\beta},Q_{6,\beta}\rbrace$ also satisfy this inequality with equality. This inequality corresponds with the green faces in Figures \ref{fig:thmmain}, \ref{fig:thmLS}, and \ref{fig:d2LS}. 

A Knapp example similar to the previous one, adapted to Assouad spectrum in the fashion we used in Section \ref{sec:assouadKnapp} gives a more refined condition. We choose our set $E$ similar to the above but where the translation invariant structure is realised on subintervals of width $\delta^\theta$. For certain sets $E$ with Minkowski dimension $\beta$ and Assouad spectrum $\gamma_\theta$. 
As $\delta\rightarrow 0$ we obtain the necessary condition
\[-\frac{(1-\theta)\gamma_\theta}{r}+\frac{(1-\theta)(d-1)}{2}+\frac{1-(1-\theta)\gamma_\theta}{q}+\frac{(d-1)(1+\theta)}{2q}\geq -\frac{(1-\theta)\gamma_\theta}{p}+\frac{1}{p}+\frac{(1-\theta)(d-1)}{2p}.\]
In the Assouad-regular case $\gamma_\theta=\gamma$, where $\theta=1-\beta/\gamma$, the vertices  $(1/p,1/q,1/r)\in \lbrace Q_{4,\gamma},Q_{D,\beta,\gamma},Q_{C,\beta},Q_{3,\beta}\rbrace$ satisfy this inequality with equality. 

\begin{appendix}

\section{Local smoothing}\label{sec:LS}
In this section, we use the notation $\mathcal{L}^p_s$ to denote the Bessel potential space of functions $f$ which are $L^p$-integrable with respect to the frequency multiplier $(1+|\xi|^2)^{-s/2}$, i.e.
\[\int\left|\int e^{i x\cdot \xi} (1+|\xi|^2)^{-s/2}\hat{f}(\xi)\right|^pdx<\infty.\]
However, as previously we denote $\|f(x)\|_{L^p_x}^p=\int |f(x)|^pdx$. The presence of integration variables and use of calligraphic $\mathcal{L}$ should distinguish between uses.

We have the following fixed time estimate from  \cite{seegerSoggeStein91}.
\begin{theorem}
With  $\bar{s}_p=\frac{d-1}{2}+\frac{1}{p}-\frac{d}{p}$,
\[\left\|e^{ \pm i t \sqrt {-\Delta}}f\right\|_{\mathcal{L}^p_{-\bar{s}_p}}\lesssim \|f\|_{L^p},\]
for $t\in[1,2]$.
\end{theorem}

The local smoothing conjecture was formulated in \cite{Sogge91}.
\begin{conjecture}For $\sigma<\frac{1}{p}$ and $\frac{2d}{d-1}\leq p < \infty$ and $\sigma < \bar{s}_p$ if $2<p\leq \frac{2d}{d-1}$
\[\left(\int_1^2 \left\|e^{ \pm i t \sqrt {-\Delta}}f\right\|_{\mathcal{L}^p_{-\bar{s}_p+\sigma}}^p dt \right)^{\frac{1}{p}} \lesssim \|f\|_{L^p}.\]
\end{conjecture}

In this context, we make a slight reformulation of Conjecture \ref{conj:roughLSMeasure}.
\begin{conjecture}[Local smoothing for rough time-averages]\label{conj:roughSmoothing}
Let $\mu$ be Ahlfors--David-regular measure with $\supp \mu \subset [1,2]$ such that $\mu(B(t,r))\sim r^\beta$ for all balls $B(t,r)$ for $r<\diam\left(\supp \mu\right)$, 
\[\left(\int_1^2\left\|e^{ \pm i t \sqrt {-\Delta}}f\right\|_{\mathcal{L}^p_{-\bar{s}_p+\sigma}}^p d\mu(t) \right)^{\frac{1}{p}} \lesssim \|f\|_{L^p},\]
for $\sigma<\frac{\beta}{p}$ and $\frac{2(d+\beta-1)}{d-1}\leq p < \infty$ and $\sigma < \bar{s}_{p}$ if $2<p\leq \frac{2(d+\beta-1)}{d-1}$.
Furthermore, the constant can be taken uniformly over the class of all such measures.
\end{conjecture}

A bounded $\mathcal{H}^s$ measurable set $E$ with $\mathcal{H}^s(E)\in (0,\infty)$ is Ahlfors--David $s$-regular if for all $x\in E$ and $r<\diam E$,
\[\frac{1}{C}r^s\leq \mathcal{H}^s|_E(B(x,r))\leq C r^s,\]
for some $C>0$.

If $E$ is Ahlfors--David regular, we can usefully implement this result with respect to our present problem. We consider the discrete probability measure
$\mu_{j} = \frac{1}{|\mathcal{Z}_j(E)|}\sum_{\nu\in\mathcal{Z}_j(E)}\delta_{t_\nu}$. Taking a weak limit (along a subsequence) we obtain some $\mu$.
By our regularity assumptions, for $r\geq 2^{-j}$, and $t\in \supp \mu$, 
$\mu_{j}(B(t,r)) \sim \mu(B(t,r))$. For less regular measures, in particular for measures with distinct Minkowski and Assouad dimensions, the relationship between the $\mu_{j}$ and $\mu$ may be less well-behaved. This regularity, together with uncertainty principle heuristics, will allow us to relate variational embedding estimates to the local smoothing estimates.

Informally, for $\supp \hat{f}\subset B(0,2^j)$, we might expect $\left(e^{ i t \sqrt {-\Delta}}f\right)(x)$ to be roughly constant over time-intervals of width $2^{-j}$. To formalise this heuristic, we take a further Littlewood--Paley decomposition in the frequency variables dual to time and use a variant of uncertainty principle arguments. Reformulating Lemma 2.5 of \cite{beltranOberlinRoncalSeegerStovall22} in the context of our embedding tells us the operator $\mathcal{A}_j$ is concentrated on time-frequency scales $\approx 2^j$. We denote by $\Lambda_l$ time-frequency projection onto scales $\sim 2^l$. Recall for 
\[\mathcal{A}_{j}f(x,t)=2^{-\frac{j(d-1)}{2}}(2\pi)^{-d-1}\sum_{\pm}\int \kappa_{j}^{\pm}(y,t) f(x-y)dy,\]
we have the kernel bound
\[|\kappa_{j}^{\pm}(x,t)|\lesssim_{N} 2^{jd}(1+2^j|x|)^{-\frac{d-1}{2}}(1+2^j||x|-|t||)^{-N}.\]
We write
\[\Lambda_l\mathcal{A}_{j}^{\pm}f(x,t)=2^{-\frac{j(d-1)}{2}}(2\pi)^{-d-1}\sum_{\pm}\int \kappa_{j,l}^{\pm}(y,t) f(x-y)dy,\]
with
\[\kappa_{j,l}^{\pm}(y,t) =\int_\RR \int_{\RR^d}e^{i\left( y\cdot \xi + t \tau\right)}\phi_l(\tau)\phi_j(|\xi)\int_{\RR} \chi(s)a_{\pm}(s,\xi)e^{is\left(\pm|\xi|-\tau\right)}dsd
\xi d\tau.\]
\begin{lemma}\label{lem:timeFrequencyErrors} \begin{enumerate}[(i)]
    \item For all $N\in \NN_0$ and $|j-l|\geq 10$,
    \[|\kappa_{j,l}^{\pm}(y,t)|\lesssim_N ( 1 + |y| + |t| )^{-N}\min\lbrace 2^{-jN},2^{-lN}\rbrace.\]
    \item Suppose that $1\leq r \leq \infty$ and $1\leq p \leq q \leq \infty$. Then there exists $C_N(p,q,r)$
for which
\[\left\|\Lambda_l\mathcal{A}_j f(x,t_\nu+h)\right\|_{L^q_x(\ell^r_\nu(\mathcal{Z}_j))}\lesssim C_N(p,q,r)\min \lbrace 2^{-jN},2^{-lN}\rbrace \left\|f\right\|_{L^p},\]
for $|j-l|\geq 10$ and $0\leq h \leq 2^{-j}$.
\end{enumerate}
\end{lemma}
We require the following technical lemma in order to preserve the time-localisation required to apply local smoothing estimates. 
\begin{lemma}\label{lem:localisedTimeFrequencyProjection}
Let $|j-l|<10$. We write
\[\Lambda_l\mathcal{A}_j^{\pm} f(x,t)=\widetilde{\Lambda}_l\mathcal{A}_j^{\pm} f(x,t)+2^{-\frac{j(d-1)}{2}}(2\pi)^{-d-1}\int E_{j,l}^{\pm}(x-y,t)f(y)dy,\]
where
\[\widetilde{\Lambda}_l\mathcal{A}_j^{\pm} f(x,t) = 2^{-\frac{j(d-1)}{2}}(2\pi)^{-d-1}\int \tilde{\kappa}_{j,l}^{\pm}(x-y,t)f(y)dy,\]
with
\[\tilde{\kappa}_{j,l}^{\pm}(y,t)=\kappa_{j}^{\pm}(y,\cdot)*\tilde{\phi_l}(t),\]
where $\tilde{\phi_l}$ is the restriction of $\check{\phi}_l$ to $[-2^{-l(1-\epsilon)},2^{-l(1-\epsilon)}]$.
\begin{enumerate}[(i)]
    \item The kernel $E_{j,l}^{\pm}(y,t)\neq 0$ only for $t$ such that $d(t,\supp\chi)\geq 2^{-j(1-\epsilon)}$. Furthermore,
    \[\left| E_{j,l}(y,t) \right|\lesssim_{N,\epsilon} 2^{-jN}\left(1+|t|+|x|\right)^{-N}.\]
    \item Suppose that $1\leq r \leq \infty$ and $1\leq p \leq q \leq \infty$. Then there exists $C_{N,\epsilon}(p,q,r)$
    for which
    \[\left\|\left(\Lambda_l\mathcal{A}_j^{\pm} -\widetilde{\Lambda}_l\mathcal{A}_j^{\pm}\right)f(x,t_\nu+h)\right\|_{L^q_x(\ell^r_\nu(\mathcal{Z}_j))}\lesssim C_{N,\epsilon}(p,q,r) 2^{-jN} \left\|f\right\|_{L^p},\]
    for $0\leq h \leq 2^{-j}$. 
    \item We have that $\|\tilde{\phi_l}(t)\|\lesssim 1$.
\end{enumerate}
\end{lemma}
\begin{proof}
This is a simple consequence of the kernel bound
\[|\kappa_{j}^{\pm}(y,t)|\lesssim_N 2^{jd} (1+ 2^j|x|)^{-(d-1)/2}\left(1+2^j||x|-|t||\right)^{-N}\]
and the representation $\kappa_{j,l}^{\pm}(y,t)=\kappa_{j}^{\pm}(y,\cdot)*\check{\phi_l}(t)$. 
\end{proof}

\begin{proposition}\label{prop:roughLSImpliesVarLS}
Suppose that $E\subset [1,2]$ is an Ahlfors--David regular set of dimension $\beta$ and that Conjecture \ref{conj:roughSmoothing} holds. Let $\nu\in\mathcal{Z}_j(E)$ index a minimal set of $2^{-j}$-separated $t_\nu$ such that the intervals $[t_\nu,t_\nu+2^{-j}]$ cover $E$. Then, for $0\leq h \leq 2^{-j}$, 
\[\left\|\mathcal{A}_{j}f(x,t_\nu+h)\right\|_{L^{q_{\mathrm{LS},\beta}}_x(\RR^d;\ell^{q_{\mathrm{LS},\beta}}_\nu(\mathcal{Z}_j(E)))}\lessapprox 2^{-\frac{j(d-1)}{2}}|\mathcal{Z}_j(E)|^{1/ q_{\mathrm{LS},\beta}}\|f\|_{L^{q_{\mathrm{LS},\beta}}}.\]
\end{proposition}
\begin{proof} Using Lemmas \ref{lem:timeFrequencyErrors} and \ref{lem:localisedTimeFrequencyProjection}, we have that
\[|\mathcal{Z}_j(E)|^{-1/ q_{\mathrm{LS},\beta}}\left\|\widetilde{\Lambda}_l\mathcal{A}_j f(x,t_\nu)\right\|_{L^{q_{\mathrm{LS},\beta}}_x(\ell^{q_{\mathrm{LS},\beta}}_\nu(\mathcal{Z}_j))}\]
\[\lesssim 
\sum_{|j-l|\leq 10 }|\mathcal{Z}_j(E)|^{-1/ q_{\mathrm{LS},\beta}}\left\|\widetilde{\Lambda}_l\mathcal{A}_j f(x,t_\nu)\right\|_{L^{q_{\mathrm{LS},\beta}}_x(\ell^{q_{\mathrm{LS},\beta}}_\nu(\mathcal{Z}_j))}+\sum_{j,l} \min\lbrace 2^{-jN},2^{-lN}\rbrace\|f\|_{L^{q_{\mathrm{LS},\beta}}}.\]
In the following, we analyse terms from the first sum with $|j-l|\leq 10$. We write $I_\nu=[t_\nu,t_\nu+2^{-j}]$. We have
\[|\mathcal{Z}_j(E)|^{-1/ q_{\mathrm{LS},\beta}}\left\|\widetilde{\Lambda}_l\mathcal{A}_j^{\pm} f(x,t_\nu+h)\right\|_{L^{q_{\mathrm{LS},\beta}}_x(\ell^{q_{\mathrm{LS},\beta}}_\nu(\mathcal{Z}_j))}\]
\[=\left(\int\int \left|\widetilde{\Lambda}_l\mathcal{A}^{\pm}_jf(x,t+h)\right|^{q_{\mathrm{LS},\beta}} dx d\mu_j(t) \right)^{\frac{1}{q_{\mathrm{LS},\beta}}}\]
\[\leq\left(\int\sum_{\nu\in\mathcal{Z}_j(E)} \sup_{t\in I_\nu}\left|\widetilde{\Lambda}_l\mathcal{A}^{\pm}_jf(x,t)\right|^{q_{\mathrm{LS},\beta}}  \mu_j(I_\nu) dx \right)^{\frac{1}{q_{\mathrm{LS},\beta}}}\]
\[\sim\left(\int\sum_{\nu\in\mathcal{Z}_j(E)} \sup_{t\in I_\nu}\left|\widetilde{\Lambda}_l\mathcal{A}^{\pm}_jf(x,t)\right|^{q_{\mathrm{LS},\beta}}  \mu(I_\nu) dx \right)^{\frac{1}{q_{\mathrm{LS},\beta}}}.\]
We define $\eta_l(s)=\sup_{|h|\leq 2^{-j}}\left|\tilde{\phi}_l(s+h)\right|$ to bound the previous expression by
\[\left(\int\sum_{\nu\in\mathcal{Z}_j(E)} \inf_{t\in I_\nu}\left(\eta_l*|\mathcal{A}^{\pm}_jf(x,\cdot)|(t)\right)^{q_{\mathrm{LS},\beta}}  \mu(I_\nu) dx \right)^{\frac{1}{q_{\mathrm{LS},\beta}}}\]
\[\lesssim\left(\int\int \left(\eta_l*|\mathcal{A}^{\pm}_jf(x,\cdot)|(t)\right)^{q_{\mathrm{LS},\beta}}  d\mu(t) dx \right)^{\frac{1}{q_{\mathrm{LS},\beta}}}\]
\[\lesssim\left(\int\int \left(\eta_l*|\mathcal{A}^{\pm}_jf(x,\cdot)|^{q_{\mathrm{LS},\beta}}(t)\right)  d\mu(t) dx \right)^{\frac{1}{q_{\mathrm{LS},\beta}}}\]
\[\lesssim\left(\int\int \left|\mathcal{A}^{\pm}_jf(x,t)\right|^{q_{\mathrm{LS},\beta}}  d\mu^l(t) dx \right)^{\frac{1}{q_{\mathrm{LS},\beta}}},\] 
where $\mu^l=\mu*\eta_l$. The measure $\mu^l$ is such that $\mu^l(B(t,r))\lesssim r^\beta$ for all balls so we can apply the local smoothing inequality to obtain the desired result.
\end{proof}

\section{The Besov embedding}\label{app:besov}
\begin{theorem}
Suppose that $F\in \mathcal{S}(\RR)$. 
Then we have that 
\[V^r_{[1,2]}F\lesssim \sum_{j\geq 0}\left(2^{\frac{j}{r}}\left\|\check{\phi_j}* F\right\|_{L^r}\right),\]
where, for $j\geq 1$, $\phi_j(\tau)=\phi(2^{-j}\tau)$ for some $\phi$ such that $\supp\phi\subset [-4,-1/2]\cup [1/2,4]$ with $\phi(\tau)=1$ for $|\tau|\in[1,2]$ and $\phi_0(\tau)=1-\sum_{j\geq 1}\phi_j(\tau)$.
\end{theorem}

\begin{proof}
We proceed as in \cite{peetre76} and \cite{guoRoosYung20}. For large $l\in\mathbb{N}$, suppose that $\supp\hat{F}\subset B(0,2^l)$. Let $\mathcal{I}_l=\lbrace I_1,I_2,\ldots I_J\rbrace$ be an ordered cover of $E$ by $\delta$-intervals, with $\delta=2^{-l}$. We can assume the left endpoints of the intervals are $\delta$-separated. Suppose that we have a sequence $t_1<t_2<\ldots<t_N$ of points in $E$. Starting from $n_0=0$, we find $n_1,n_2,\ldots,n_J$ and $N(k)=\sum_{j=1}^{k}n_j$ such that for $N(k-1)\leq i < N(k)$, $t_i\in I_k=[a_k,a_k+\delta]$. By multiple applications of the triangle inequality, we find that 
    \[
    \left(\sum_{i=2}^N|F(t_i)-F(t_{i-1})|^r\right)^\frac{1}{r}
    \]
    \[
    =\left(\sum_{k=1}^{J}\sum_{i=N(k-1)+1}^{N(k)}|F(t_i)-F(t_{i-1})|^r+\sum_{k=2}^{J}|F(t_{N(k)})-F(t_{N(k)-1})|^r\right)^\frac{1}{r}
    \]
    \begin{equation}\label{eq:BesovEmbFTCterms}
    \lesssim\left(\sum_{k=1}^{J}\sum_{i=N(k-1)}^{N(k)}|F(t_i)-F(t_{i-1})|^r\right)^\frac{1}{r}
    \end{equation}
    \[
    +\left(\sum_{k=2}^{J-1}|F(t_{N(k)})-F(a_{k+1})|^r\right)^\frac{1}{r}
    \]
    \[
    +\left(\sum_{k=2}^{J-1}|F(a_{k})-F(t_{N(k)-1})|^r\right)^\frac{1}{r}
    \]
    \begin{equation}\label{eq:BesovEmbPPterms}
   +\left(\sum_{k=2}^{J-1}|F(a_{k+1})|^r+|F(a_{k})|^r\right)^\frac{1}{r}.
    \end{equation}
    The sum in \eqref{eq:BesovEmbFTCterms} is bounded by an application of the Fundamental Theorem of Calculus and Bernstein's inequality, as for the remaining sums, excepting \eqref{eq:BesovEmbPPterms}, which is bounded by the Plancherel--Polya inequality. Applying H\"{o}lder's inequality followed by Bernstein's inequality, we find that 
    \[\left(\sum_{k=1}^{J}\sum_{i=N(k-1)}^{N(k)}|F(t_i)-F(t_{i-1})|^r\right)^\frac{1}{r}\]
    \[\leq\left(\sum_{k=1}^{J}\sum_{i=N(k-1)}^{N(k)}(t_i-t_{i-1})^{r-1}\left(\int_{t_{i-1}}^{t_i}|F'(t)|^r dt \right)\right)^\frac{1}{r}\]
    \[\leq \delta^{1-\frac{1}{r}}\|F'\|_{L^r}\]
    \[\lesssim_{r} \delta^{-\frac{1}{r}}\|F\|_{L^r}.\]
    The complete result now follows by an application of the triangle inequality to the Littlewood-Paley decomposition of $F=\sum_{j\geq 0}\check{\phi_j}*F$.
    \end{proof}

\end{appendix}

\bibliographystyle{plain}
\bibliography{fractBib}
\end{document}